\documentclass[11pt,english]{article}
\usepackage[T1]{fontenc}
\usepackage[utf8]{inputenc}
\usepackage{mathrsfs, amsmath, amstext, amsthm, amssymb}
\usepackage[all]{xy}
\usepackage{mathtools}

\usepackage{enumitem}
\usepackage{babel}
\usepackage{caption}
\makeatletter

\numberwithin{equation}{section}
\numberwithin{figure}{section}

\theoremstyle{plain}
\newtheorem{thm}{\protect\theoremname}[section]
\theoremstyle{definition}
\newtheorem{defn}[thm]{\protect\definitionname}
\theoremstyle{remark}
\newtheorem{rem}[thm]{\protect\remarkname}
\theoremstyle{remark}

\theoremstyle{plain}
\newtheorem{lem}[thm]{\protect\lemmaname}
\theoremstyle{plain}
\newtheorem{cor}[thm]{\protect\corollaryname}
\theoremstyle{remark}
\newtheorem{ex}[thm]{\protect\examplename}
\theoremstyle{plain}
\newtheorem{prop}[thm]{\protect\propositionname}
\theoremstyle{definition}
\theoremstyle{definition}
\newtheorem{note}[thm]{\protect\notename}
\theoremstyle{remark}
\newtheorem{construction}[thm]{\protect\constructionname}
\providecommand{\constructionname}{Construction}
\theoremstyle{plain} 
\newtheorem{conj}[thm]{Conjecture}
\theoremstyle{definition}
\newtheorem{problem}{Problem}

\newtheorem*{thmstar}{\protect\theoremname}

\usepackage{graphicx}
\usepackage{wrapfig}
\usepackage{fancyhdr}
\usepackage{geometry}
\usepackage{hyperref}
\usepackage{url}
\makeatother

\title{Exceptional Singularities in Polynomial Neural Networks}
\author{Giacomo Graziani\thanks{Department of Civil, Environmental and Architectural Engineering, University of Padua, Via F. Marzolo 9, 35131 Padova, Italy.
Email: \texttt{giacomo.graziani@unipd.it}}}
\date{}

\providecommand{\definitionname}{Definition}
\providecommand{\examplename}{Example}
\providecommand{\lemmaname}{Lemma}
\providecommand{\corollaryname}{Corollary}
\providecommand{\propositionname}{Proposition}
\providecommand{\remarkname}{Remark}
\providecommand{\theoremname}{Theorem}
\providecommand{\notename}{Note}

\global\long\def\PP{\mathbb{P}}%
\global\long\def\CC{\mathbb{C}}%

\global\long\def\G{\mathcal{G}}%
\global\long\def\rr{\boldsymbol{\underline{r}}}%
\global\long\def\d{\underline{\boldsymbol{d}}}%
\global\long\def\a{\boldsymbol{\underline{\alpha}}}%

\global\long\def\p{^{\prime}}%
\global\long\def\c{^{\circ}}%
\global\long\def\e{\emptyset}%

\DeclarePairedDelimiter\pInternal{(}{)}

\newcommand{\pr}[1]{\pInternal*{#1}}

\global\long\def\F{\mathcal{F}}%
\global\long\def\E{\mathcal{E}}%
\global\long\def\OO{\mathcal{O}}%
\global\long\def\X{\mathcal{X}}%
\global\long\def\M{\mathcal{M}}%
\global\long\def\N{\mathcal{N}}%
\global\long\def\W{\mathcal{W}}%

\global\long\def\ot{\otimes}%
\global\long\def\inv{^{\ast}}%
\global\long\def\su{\subseteq}%
\global\long\def\i{_{\mathrm{in}}}%
\global\long\def\o{_{\mathrm{out}}}%
\global\long\def\n{_{\mathrm{norm}}}%

\global\long\def\ga{\alpha}%
\global\long\def\gb{\beta}%
\global\long\def\gs{\sigma}%
\global\long\def\gl{\lambda}%

\global\long\def\g{\mathfrak{g}}%
\global\long\def\Om{\mathfrak{W}}%
\DeclareMathOperator{\Sym}{Sym}

\DeclareMathOperator{\Sing}{Sing}
\global\long\def\Hom{\mathrm{Hom}}%
\DeclareMathOperator{\Gr}{Gr}
\global\long\def\Da{\mathcal{D}_{\mathrm{arch}}}%
\DeclareMathOperator{\Lie}{Lie}
\DeclareMathOperator{\rk}{rk}

\begin{document}

\maketitle

\begin{abstract}
The singularities of a neurovariety have so far always turned out to lie at subnetworks: functions already computed by a smaller architecture of the same kind.  This is what makes them readable as a bias of training towards simpler models, and whether they exhaust the singular locus is left open in the literature. We show that they do not.  There are singular points at which every neuron is active, no two neurons are redundant, and the parametrisation is unramified, so that nothing visible at a single parameter signals the singularity.  We call these exceptional, and we show that, after quotienting by the group of architectural symmetries, which acts freely on the full parameters, they can arise in only three ways: a rank drop of the differential of the parametrisation, a disconnected fibre, or a failure of properness.  A small dimensional example is a fully connected network with two hidden layers of width two and quadratic activations. All the examples we produce are of the second kind, so that what fails at an exceptional singularity is not the simplicity of the model but the possibility of recovering the parameter, up to symmetry, from the function it computes.
\end{abstract}

\tableofcontents 
\newpage
\section*{Introduction}
The functions computed by a neural network with polynomial activations form a semialgebraic subset of a finite-dimensional real vector space of polynomial maps. Its Zariski closure is an algebraic variety, the neurovariety of the architecture. Starting with \cite{KTB19}, this variety has become the object through which the architecture itself is studied: for example, its dimension measures the expressivity of the model and controls its complexity, while the learning degree introduced in \cite{KLW24} measures the complexity of training. The programme of deducing properties of a machine-learning model from the geometry of its associated variety was named neuroalgebraic geometry in \cite{MSMTK25}, and is now well under way for several classes of networks: fully connected networks \cite{KTB19,KLW24,UBDC25,MM25}, convolutional networks \cite{KMMT22,SMK25,Sha24}, attention mechanisms \cite{HMK25}, and equivariant architectures \cite{SMBK26}.

Among the features of these neurovarieties, singularities play a special role because they are linked to learning biases: they can attract the dynamics and are therefore thought to introduce a data-independent implicit bias into training. This viewpoint originates in singular learning theory \cite{Wat09} and is discussed, in the language of the present paper, in \cite[Section 3.2]{MSMTK25}.

Whenever the singular locus of a neurovariety has been determined, the same picture has emerged: it consists of subnetworks, that is, of functions already computed by a smaller architecture of the same type. For a fully connected linear network, the neurovariety is a determinantal variety and, classically, its singular points are the lower-rank maps; see \cite[Appendix A]{MSMTK25} and Example \ref{ex:LineareEVuoto}. For a convolutional network with monomial activation of degree at least two, the singular points are exactly the zero function and the functions admitting more than one orbit of parameters, and the latter are precisely those that can be reparametrised using smaller filters \cite[Corollary 4.8 and Remark 4.2]{SMK25}. An analogous characterisation holds, away from the zero function and under an arithmetic condition on the shifts, for a generic polynomial activation of sufficiently large degree that vanishes at the origin \cite[Theorem 4.6]{SMK26}. For a single layer of Lightning Self-Attention, the singular locus is again defined by a rank condition on the weights \cite[Theorem 3.4]{HMK25}. The picture on the parameter-space side is consistent with this: for monomial networks with large activation degree, the parametrisation is critical exactly at subnetworks \cite[Theorem 5.1]{KMSSW26}. For fully connected networks with non-linear activation, however, only one implication is known, and under assumptions on the activation and on the neurons being removed: these subnetworks are singular points \cite[Theorem 4.2]{SMK26}. Whether they exhaust the singular locus is left as an open problem in the same paper, and a general relation between singularities and subnetworks is observed in \cite{MSMTK25}, where the authors state that singularities of neurovarieties often correspond to subnetworks.

This paper answers that question, and the answer is negative. To state this precisely, we first need to isolate what should be excluded. 
Given an algebraic network $X$ with parameter space $\W$ and realisation map
$\Phi_X:\W\to X\o\ot\Sym^R \left(X\i\inv\right)$, we say that a parameter
$w=\left(w_1 ,\dots,w_{L}\right)$ is full if every $w_i$ has maximal rank, if no
hidden neuron has a vanishing incoming row or a vanishing outgoing column, and
if no two hidden neurons in the same layer have proportional incoming rows
(Definition \ref{def:LuogoDegenere}); in particular, in a full parameter every
neuron is active and no two neurons in the same layer are redundant, in the
sense of \cite{KMSSW26}. We define the architectural degeneration locus $\Da\left(X\right)$
as the closure of the set of functions realised by a parameter that is not full.
For a fully connected network, with no hypothesis on the widths, these are
precisely the functions already realised by a proper subnetwork of $X$, where a
subnetwork is obtained by deleting hidden neurons and then imposing rank bounds
on the weights (Proposition \ref{prop:DarchSottoreti}); this notion is strictly
broader than the one in \cite{SMK26}, which only deletes neurons.  We call a
network reduced if $\F\left(Z\right)\neq\F\left(X\right)$ for every proper
subnetwork $Z$, equivalently if $\Da\left(X\right)\neq\F\left(X\right)$. We then define
\[
\E\left(X\right)=\left(\Sing\F\left(X\right)\cap\Phi_X\left(\W\right)\right)\setminus\Da\left(X\right),
\]
the exceptional locus of $X$, whose points are called exceptional singularities. We then give examples of networks for which $\E\left(X\right)\neq\e$, that is, singular points of the neurovariety $\F\left(X\right)$ that are realised by some parameter but are not degenerations of the architecture (Definition \ref{def:LuogoEccezionale}). There is also a structural counterpart: it is observed in \cite[Section 3.3]{MSMTK25} that, in parameter space, the singularities of the neurovariety are resolved and are instead recorded as critical points of the parametrisation or as non-generic fibres. We prove that this observation becomes a theorem once failure of properness of the parametrisation is included among the possibilities.

It is worth pointing out that linear convolutional networks show that the full locus depends on the chosen
bases of the layers: in the monomial basis, the first two rows of a transposed
convolution of stride larger than one are both supported on the first column,
hence proportional or zero, so $\W\c=\e$ as soon as some hidden layer has stride
greater than one; for stride one the condition is satisfied (Example \ref{ex:LCNdocili}).

A small-dimensional example of exceptional singularity comes from the fully connected network with quadratic activations (Section \ref{Sec:Esempio})
\begin{equation}\label{eq:IntroEsempio}
    Y:\CC^2 \longrightarrow\CC^2 \longrightarrow\CC^2 \longrightarrow\CC^2\qquad \text{with}\qquad \rr=\left(2,2\right).
\end{equation}
Let $x,y$ be coordinates on the input space and set
\[
u=x^2 +y^2 ,\qquad v=x^2 -y^2 ,\qquad p=\left(u^2 ,v^2 \right).
\]
Writing $\G\left(Y\right)$ for the symmetry group of the architecture in the sense of \cite{UBDC25}, the fibre of $\Phi_Y$ over $p$ consists of exactly three $\G\left(Y\right)$-orbits, all made up of full parameters, and $p\notin\Da\left(Y\right)$ (Section \ref{Sec:Esempio}); the tangent images at two of these orbits span a space of dimension at least $9$, while $\dim\F\left(Y\right)=8$, so $p\in\E\left(Y\right)$. Example \eqref{eq:IntroEsempio} is not the consequence of a redundant architecture, in fact we see that $Y$ is reduced and generically finitely identifiable, and that all three parameters are good in the sense of Definition \ref{def:Parametribuoni}, so the parametrisation is unramified at every point over $p$ (Remark \ref{rem:NonRidondante}), and nothing visible in the parameters seems to signal that $p$ is a singular point of $\F\left(Y\right)$.

To try to generalise \eqref{eq:IntroEsempio}, we introduce a reduced parameter space. Indeed, we observe that the symmetry group $\G\left(X\right)$ is linearly reductive, so the affine quotient $\Om\left(X\right) = \mathrm{Spec}\left(\CC\left[\W\right]^{\G\left(X\right)}\right)$, whose points are the closed orbits, is well defined (Theorem \ref{thm:RiduttivitaLineare}) and gives a factorisation $\Phi_X=\phi_X\circ\pi_X$. As soon as $\W\c$ is non-empty and contained in the smooth locus of $\W$ (Condition \eqref{cond:ReteDocile}), the image $\Om\c\left(X\right)=\pi_X\left(\W\c\right)$ is a smooth open subset of $\Om\left(X\right)$ and $\pi_X:\W\c\to\Om\c\left(X\right)$ is a $\G\left(X\right)$-torsor locally trivial in the \'etale topology (Theorem \ref{thm:OmegaLiscio}).
The condition asks nothing about the group action: on the full locus the action
of $\G\left(X\right)$ is free and its orbits are closed for every algebraic
network, with no hypothesis on the widths, precisely because a full parameter
has no vanishing and no pairwise proportional incoming rows
(Remark \ref{rem:AzioneLibera}).  What it does ask is that $\W\c$ be non-empty
and contained in the smooth locus of $\W$; in particular it holds for every
fully connected network satisfying \eqref{eq:CondizioneEsisteParametroPieno}
(Corollary \ref{cor:OmLiscioTotalmenteConnesse}).

Since the whole fibre of $\phi_X$ over a point outside $\Da\left(X\right)$ is contained in $\Om\c\left(X\right)$, an exceptional singularity cannot be blamed on the source of the parametrisation: it must be caused by the map $\phi_X$.

In Theorem \ref{thm:Singolarita} we show that exceptional singularities for networks satisfying \eqref{cond:ReteDocile} can come from only three sources. Following \cite[Appendix A]{SMK26}, in Definition \ref{def:Fonti} we say, with a slight abuse of terminology, that $p$ is a cusp if some parameter in its fibre satisfies $\rk d\Phi_{X,w}<\dim\F\left(X\right)$, that it is a node if $\phi_X^{-1}\left(p\right)$ is not connected, and we write $C\left(X\right)$, $N\left(X\right)$, and $P\left(X\right)$ for the sets of cusps, nodes, and points around which $\phi_X$ is not proper, respectively. Then (Theorem \ref{thm:Singolarita})
\[
\Sing\F\left(X\right)\ \su\ \Da\left(X\right)\cup C\left(X\right)\cup N\left(X\right)\cup P\left(X\right),
\]
and it follows immediately that every exceptional singularity exhibits at least one of the three behaviours (Corollary \ref{cor:Fonti}). The proof is analytic and local: properness makes the fibre compact; the fact that the point is not a cusp, together with the constant rank theorem, makes the image of a neighbourhood of each point of the fibre a smooth germ at $p$; and connectedness makes these germs coincide. There is also a pointwise counterpart: if $p\notin\Da\left(X\right)$ is rigid, if every parameter over it is good, and if $\OO_{\Om\left(X\right),\left[w\right]}$ is a finite $\OO_{\F\left(X\right),p}$-module, then the two local rings are isomorphic and $p$ is a smooth point (Theorem \ref{thm:PuntualeProprio}); the last condition holds everywhere as soon as $\phi_X$ is finite, a property that, by Proposition \ref{prop:CriterioFinitezza}, is equivalent to instability of the parameters realising the zero function.

We then investigate what happens when $P\left(X\right)=\e$, that is, when the parametrisation map is proper (Subsection \ref{Sub:Propria}). Since $\Om\left(X\right)$ and $\F\left(X\right)$ are both affine, $\phi_X$ is proper exactly when it is finite and, for a conical parameter space, finiteness reduces to the fibre over the zero function: $\phi_X$ is finite if and only if every parameter realising the zero function is unstable for the action of $\G\left(X\right)$ (Proposition \ref{prop:CriterioFinitezza}). This criterion sharply separates different architectures. One consequence is that a fully connected network with at least two neurons in its last hidden layer is never finitely identifiable: two outputs can be made equal and then cancelled against each other by the last layer, producing infinitely many orbits over the zero function (Proposition \ref{prop:FCnonPropria}). By contrast, a constrained parameter space may leave no room for such a cancellation, and linear convolutional networks provide an example: in that case $\Phi_X$ is a product of polynomials, its zero locus is the union of the sets $\left\{ f_i=0\right\}$, and each of its points is unstable for every stride (Remark \ref{rem:DicotomiaArchitetture}). 
For stride one the conclusion is vacuous, because $\F\left(X\right)$ is then a linear space over $\CC$.  The symmetry group is the rescaling torus for every choice of the strides (Subsection \ref{Sec:LCNCalcoloGruppo}), while the full locus is non-empty exactly when every hidden stride equals one, the last stride being unrestricted; taking $s_{L}>1$ therefore gives a non-linear neurovariety with a proper parametrisation and a well-behaved quotient on the full locus (Remark \ref{rem:DicotomiaArchitetture}).

Building on the previous sections, we then generalise the example of Section \ref{Sec:Esempio} and analyse the shallow fully connected cases. We call an index $t$ a quadratic block of $X$ if $r_t=r_{t+1}=2$ and $m_t=m_{t+1}=m\ge2$. For every fully connected network with constant hidden widths, including the output layer and input dimension at least $m$, and a quadratic block, we construct three full parameters belonging to three distinct $\G\left(X\right)$-orbits and realising the same function $p\notin\Da\left(X\right)$. To conclude that $p\in\E\left(X\right)$ in the spirit of Section \ref{Sec:Esempio}, it remains to check that the three tangent images span a space of dimension greater than $\dim\Om\left(X\right)$, which would imply that $p$ is exceptional (Theorem \ref{thm:Generalizzazione}). We verify this condition in exact arithmetic for nineteen architectures. In these cases, the excess seems to depend only on the width $m$ of the block and to be equal to $5m-8$ throughout the range $2\le m\le7$: it also seems to be independent of the input dimension, the depth, and the activation degrees outside the block (Remark \ref{rem:CasiVerificati}). Whether this holds in general is Problem \ref{prob:Eccesso}, although even the weaker statement that the excess is positive would be enough to make Theorem \ref{thm:Generalizzazione} unconditional.
The mechanism we use to produce these examples is specifically quadratic: by Fermat's Last Theorem for polynomials, the pencil $\left\langle A^r ,B^r \right\rangle$ contains no further $r$-th powers when $r\ge3$ (Remark \ref{rem:SoloGradoDue}). Together with \cite[Theorem 5.1]{KMSSW26}, this suggests Conjecture \ref{conj:Soglia}: for every sequence of widths there is a threshold beyond which all activation degrees give $\E\left(X\right)=\e$.

We then show that every architecture listed in Remark \ref{rem:CasiVerificati} is reduced, generically finitely identifiable, and non-defective. To do this, we apply Corollary \ref{cor:CriterioRango}, where we show, without elimination or Gr\"obner bases, that for a fully connected network with non-empty full locus, the rank of $d\Phi_{X,w}$ at a single parameter, when maximal, simultaneously certifies that the architecture is reduced and that its parametrisation has finite general fibres.

In the shallow case, the formalism meets the classical theory of secant varieties: the neurovariety of a shallow network is the affine cone over a secant variety of a Segre--Veronese variety (Example \ref{ex:ShallowSecante}). Exceptional singularities seem to be rare for these networks: in Remark \ref{rem:SottosecantiInDarch} we check that the subsecant loci of \cite{FH25} and the $\left(d-1\right)$-st secant variety are all contained in $\Da\left(X\right)$ and therefore cannot produce exceptional singularities. We also show (Proposition \ref{prop:ShallowMonomorfismo}) that, in the regime \eqref{eq:NonCrescenteShallow}, and for $r\ge3$ or for $r=2$ with $\dim\N=\dim X\o$, the parametrisation produces neither cusps nor nodes, and the only source of exceptional singularities is failure of properness. In Theorem \ref{thm:ListaShallow} we collect the cases in which we can prove that $\E\left(X\right)=\e$.

\vspace{0.25cm}

Outside the cases just mentioned, the picture changes. When $d>\dim X\i$, a full parameter need no longer be non-redundant, and the failure is not merely technical. Taking $X$ to be the shallow network $\CC^3 \to\CC^{14}\to\CC$ with $r=8$, we consider the ternary octic $T$ constructed by Angelini and Chiantini in \cite{AC20}, which admits two disjoint decompositions of length $14$. This gives a singular point of $\gs_{14}\left(v_{8}\left(\PP^2 \right)\right)$ outside $\Da\left(X\right)$, so $\E\left(X\right)\neq\e$ even for a shallow network (Example \ref{ex:ShallowEccezionale}). The existence of exceptional singularities is therefore not exclusive to depth. We also point out that all the examples of exceptional singularities in this paper are examples of nodes.

Finally, we note that $\E$ is not invariant under elementary reductions: contracting sequences of linear blocks and one-dimensional chains clearly leaves both $\F\left(X\right)$ and the image of $\Phi_X$ unchanged, and the two operations produce a well-defined normal form (Lemma \ref{lem:FormaNormaleUnica}); these operations cannot enlarge $\Da$, so $\E\left(X\right)\su\E\left(X\n\right)$, and the inclusion can be strict. For example, prefixing a three-dimensional linear layer to the network of Section \ref{Sec:Esempio} gives $\Da\left(X\right)=\F\left(X\right)$ and hence $\E\left(X\right)=\e$, while its normal form has $\E\neq\e$ (Example \ref{ex:InclusioneStretta}): the exceptional locus is therefore not an invariant of the neurovariety, and it is natural to state Conjecture \ref{conj:Soglia} for networks in normal form.

We conclude by explaining why this distinction matters beyond algebraic geometry: the three sources have different consequences for optimisation. For monomial networks of large activation degree, the parametrisation is
critical only at parameters that are not full \cite[Theorem 5.1]{KMSSW26},
which with our definition reads $C\left(X\right)\su\Da\left(X\right)$: in that
regime a cusp can never be exceptional, and an exceptional singularity is
necessarily a node or a point of non-properness. Moreover strict subnetworks are critically exposed in the sense of \cite[Theorem 4.3]{SMK26}, that is, they appear as critical points of the training loss for a non-negligible set of datasets: this is the geometric content of the implicit bias towards sparse solutions. Exceptional singularities lie outside $\Da\left(X\right)$ by definition, so that chain of implications does not apply to them: the parameters realising them are full and, in the example of Section \ref{Sec:Esempio}, good, so nothing in the parametrisation signals that the function is singular. What fails is the global injectivity of $\phi_X$ on $\Om\c\left(X\right)$, that is, the possibility of recovering the parameter, modulo the symmetries of the architecture, from the realised function. Whether these singularities are also critically exposed, and for which datasets, is a question we do not address here.
\newpage

\section{Algebraic networks and parameter spaces}\label{Sec:Setup}

All vector spaces and schemes are over $\CC$. We work with networks with monomial activations and no bias; see \cite{KTB19,KLW24,MSMTK25} for more details and the general picture.

\subsection{Networks and neurovarieties}

\begin{defn}\label{def:ReteAlgebrica}
An algebraic network is a triple $X=\left(\N,\W,\rr\right)$ where
\begin{itemize}
\item $\N=\left(\N_{0},\dots,\N_{L}\right)$ is a sequence of finite dimensional non-zero vector spaces, together with the choice of a basis $\left(e_{i}^{j}\right)$, $j=1,\dots,\dim\N_{i}$ of $\N_{i}$ for $i=1,\dots,L-1$. We write $X\i=\N_{0}$, $X\o=\N_{L}$, called the spaces of inputs and outputs respectively, and we call the remaining $\N_{i}$ the hidden layers.
\item $\W$ is an irreducible Zariski closed subset of
\begin{equation}\label{eq:SpazioParametri}
\widehat{\W}=\Hom_{\CC}\left(\N_{0},\N_1 \right)\times\dots\times\Hom_{\CC}\left(\N_{L-1},\N_{L}\right)
\end{equation}
 called the parameter space and we assume that $\left( 0,\dots,0\right)\in\W$, in particular $\W$ is non-empty. The network is said to be fully connected if $\W=\widehat{\W}$.
\item $\rr=\left(r_1 ,\dots,r_{L-1}\right)$ is a vector of integers $r_{i}\ge1$, called
the activation degrees. We set $R=r_1 \cdots r_{L-1}$.
\end{itemize}
We call $L$ the length of the network and we say that $X$ is shallow if $L=2$ and deep if $L\ge3$. Writing $\gs_{r}:\N_{i}\to\N_{i}$ for the map raising each coordinate to the $r$-th power with respect to the fixed basis, the realisation map is
\[
\begin{aligned}\Phi_X:\W & \longrightarrow X\o\ot\Sym^R \left(X\i\inv\right)\\
w=\left(w_1 ,\dots,w_{L}\right) & \longmapsto w_{L}\circ\gs_{r_{L-1}}\circ w_{L-1}\circ\dots\circ\gs_{r_1 }\circ w_1 ,
\end{aligned}
\]
and the neurovariety of $X$ is the Zariski closure
\[
\F\pr{X} =\overline{\Phi_X\left(\W\right)}\su X\o\ot\Sym^R \left(X\i\inv\right).
\]
\end{defn}

Note that $\Phi_X$ is a morphism of algebraic varieties, in particular, since $\W$ is non-empty and irreducible, so is $\F\left(X\right)$.

\begin{ex}\label{ex:Esempi}
Fully connected networks are the polynomial MLP of \cite{KTB19,KLW24}. Taking for $\W$ a product of spaces of Toeplitz matrices one recovers the polynomial convolutional networks of \cite{KMMT22,SMK25,Sha24}; other choices of $\W$ produce the architectures considered in \cite{HMK25,SMBK26}. 
\end{ex}

\begin{lem}\label{lem:ConoLiscio}
Let $X$ be a fully connected network. Then $\F\left(X\right)$ is a cone with vertex the origin, and the following are equivalent:
\begin{enumerate}
\item $\F\left(X\right)$ is smooth;
\item the origin is a smooth point of $\F\left(X\right)$;
\item $X$ is filling, that is $\F\left(X\right)=X\o\ot\Sym^R \left(X\i\inv\right)$.
\end{enumerate}
\end{lem}

\begin{proof}
First note that, for $c\in X\o$ and $\ell\in X\i\inv$, if we denote by
\begin{align*}
    \xi:&\N _{L-1}\to X\o\\
    \left(x_1,\dots,x_d\right)&\mapsto x_1 \cdot c,
\end{align*}
then the image of
\[
\left(\begin{pmatrix}
    \ell \\
    0 \\
    \vdots \\
    0
\end{pmatrix},
\begin{pmatrix}
    1&0&\dots&0 \\
    0&0&\dots&0 \\
    \vdots&\vdots&\dots&\vdots \\
    0&0&\dots&0
\end{pmatrix},
\dots,
\begin{pmatrix}
    1&0&\dots&0 \\
    0&0&\dots&0 \\
    \vdots&\vdots&\dots&\vdots \\
    0&0&\dots&0
\end{pmatrix},
\xi_c
\right)
\]
is $c\ot\ell^R$, in particular all the elementary tensors of $X\o\ot\Sym^R \left(X\i\inv\right)$ lie in $\F\left(X\right)$. Note moreover that, being $X$ fully connected, the image of $\Phi_X$ is invariant under rescaling of parameters, in particular $\F\left(X\right)$ is a cone. We conclude since a cone is smooth at its vertex if and only if it is a linear space.
\end{proof}

\begin{ex}\label{ex:ShallowSecante}
Let $X:X\i\to\N\to X\o$ be a fully connected shallow network of activation degree $r$ and let $d=\dim\N$. Denoting by $\ell_1 ,\dots,\ell_{d}$ the rows of $w_1 $ and by $c_1 ,\dots,c_{d}$ the columns of $w_2 $ we have
\[
\Phi_X\left(w\right)=\sum_{i=1}^d c_{i}\ot\ell_{i}^r ,
\]
so that $\F\left(X\right)$ is the affine cone over the $d$-th secant variety $\gs_{d}\left(\PP\left(X\o\right)\times v_{r}\PP\left(X\i\inv\right)\right)$ of a Segre-Veronese variety. Shallow networks are therefore an instance of the classical theory of secant varieties; see \cite{BCCGO18,Lan,OV25}.
\end{ex}

\subsection{Architectural degenerations}

Following \cite{KMSSW26} we give the following definition.

\begin{defn}\label{def:LuogoDegenere}
Let $X=\left(\N,\W,\rr\right)$ be an algebraic network. We denote by
$\W\c\su\W$ the set of parameters $w=\left(w_1 ,\dots,w_{L}\right)$ such that
\begin{itemize}
\item every $w_{i}$ has maximal rank,
\item for every hidden layer $t=1,\dots,L-1$ the rows of $w_{t}$ are non-zero
and pairwise non-proportional, and
\item for every hidden layer $t=1,\dots,L-1$ and every basis vector
$e_{t}^{j}\in\N_{t}$ one has $w_{t+1}\left(e_{t}^{j}\right)\neq0$,
\end{itemize}
and we call the elements of $\W\c$ the full parameters. The architectural
degeneracy locus of $X$ is
\[
\Da\left(X\right)=\overline{\Phi_X\left(\W\setminus\W\c\right)}\su\F\left(X\right).
\]
\end{defn}

The locus $\W\c$ is open in $\W$, but it can be empty even in the fully connected case, for instance the fully connected network
\[
X:\CC\to\CC^2\to\CC\qquad r\ge1
\]
has no full parameter, since every linear map $\CC\to\CC^2$ either has a vanishing row or the two rows are proportional. Note however that this is essentially the only obstruction: if $X$ is fully connected, then $\W\c\neq\e$ if and only if
\begin{equation}\label{eq:CondizioneEsisteParametroPieno}
\dim\N_{t-1}=1\quad\Longrightarrow\quad\dim\N_t=1
\qquad\text{for every }t=1,\dots,L-1.
\end{equation}
Indeed, if $\dim\N_{t-1}=1$ and $\dim\N_t>1$, any two non-zero rows of $w_t$ are proportional. Conversely, if \eqref{eq:CondizioneEsisteParametroPieno} holds, a general parameter satisfies all the conditions of Definition \ref{def:LuogoDegenere}. In particular, $\W\c$ is non-empty whenever the widths are non-increasing.

The last two conditions in Definition \ref{def:LuogoDegenere} say that every
hidden neuron is genuinely used: it receives a non-zero input, produces a
non-zero output, and is not a copy of another neuron in the same layer.  In
the terminology of \cite{KMSSW26}, a neuron is inactive when its incoming row
or its outgoing column vanishes, and redundant when its incoming row is
proportional to that of another neuron in the same layer.  Thus every
parameter exhibiting a proper subnetwork in the sense of \cite{KMSSW26} is
non-full.  The converse does not hold in general, since a matrix of
non-maximal rank need not exhibit an inactive or redundant neuron.  The
vanishing outgoing-column condition alone corresponds to the notion of
subnetwork used in \cite{SMK26}.

\subsubsection{Degenerations and reducedness for fully connected networks}

\begin{defn}\label{def:SottoreteRango}
Let $X$ be a fully connected network and let $\d=\left(d_1 ,\dots,d_{L}\right)$ be a sequence of non-negative integers. We write $X^{\d}$ for the algebraic network with the same $\N$ and $\rr$ and with parameter space
\[
\W^{\d}=\left\{ w\in\W\,\vert\,\rk\left(w_{i}\right)\le d_{i}\mbox{ for every }i\right\}.
\]
A subnetwork of $X$ is a network of the form $Y^{\d}$ where $Y$ is a fully connected network obtained from $X$ by deleting neurons from its hidden layers. We say that $Y^{\d}$ is a proper subnetwork of $X$ if 
\[
\W\left(Y^{\d}\right)\neq\W\left(X\right)
\]
where we embed $\W\left(Y^{\d}\right)\su\W\left(X\right)$ extending the matrices by 0. We say that $X$ is reduced if $\F\left(Z\right)\neq\F\left(X\right)$ for every proper subnetwork $Z$.
\end{defn}

Note that all of the above makes sense for an arbitrary network since we assumed that $\left(0,\dots,0\right)\in\W$, provided that the spaces $\W^{\d}$ are still irreducible. In the fully connected case each $\W^{\d}$ is a product of determinantal varieties, hence irreducible. 

\begin{prop}\label{prop:DarchSottoreti}
Let $X$ be a fully connected network.  Then
\[
\Da\left(X\right)=\overline{\bigcup_{Z\text{ proper}}\F\left(Z\right)},
\]
the union being over the proper subnetworks $Z$ of $X$, and $X$ is reduced if
and only if $\F\left(X\right)\neq\Da\left(X\right)$.
\end{prop}

\begin{proof}
Let $Z=Y^{\d}$ be a proper subnetwork. If $Y=X$ then we are imposing a proper rank restriction on $X$ and clearly $\F\left(Z\right)\su\Da\left(X\right)$. If $Y\neq X$ then some neuron $j$ of some hidden layer $t$ has been deleted, so that every $w\in\W\left(Z\right)$ has vanishing $j$-th row of $w_{t}$. We conclude that
\[
\overline{\bigcup_{Z\text{ proper}}\F\left(Z\right)}\ \su\ \Da\left(X\right).
\]
For the reverse inclusion let $w\in\W\setminus\W\c$; we exhibit a proper subnetwork $Z$ with $\Phi_X\left(w\right)\in\F\left(Z\right)$.  For a layer $t$ let  $q^{\left(t\right)}$ and $a^{\left(t\right)}=\gs_{r_{t}}\left(q^{\left(t\right)}\right)$ be the functions at that layer, the pre- and post-activations at the $t$-th layer, and by $c_{k}^{\left(t\right)}$ the $k$-th column of $w_{t+1}$, so that $q^{\left(t+1\right)} = \sum_{k}c_{k}^{\left(t\right)}a_{k}^{\left(t\right)}$. If some $w_{i}$ has non-maximal rank, take $Y=X$ and $\d$ with $d_{i}=\rk w_{i}$ and $d_{k}=\min\left\{ \dim\N_{k-1},\dim\N_{k}\right\} $ for $k\neq i$; then $w\in\W\left(X^{\d}\right)\neq\W$ and we are done.

If the $j$-th row of $w_{t}$ vanishes, then $q_{j}^{\left(t\right)}=0$ and hence
$a_{j}^{\left(t\right)}=0$, so that the column $c_{j}^{\left(t\right)}$ gives no contribution:
the parameter $\widetilde{w}$ obtained from $w$ by setting $c_{j}^{\left(t\right)}=0$ as well
satisfies $\Phi_X\left(\widetilde{w}\right)=\Phi_X\left(w\right)$ and lies in
$\W\left(Y\right)$, where $Y$ is obtained from $X$ by deleting the $j$-th neuron
of the $t$-th layer.

If the $j_2$-th row of $w_{t}$ is $\gl$ times the $j_1$-th, with
$j_1 \neq j_2$, then $q_{j_2}^{\left(t\right)}=\gl q_{j_1}^{\left(t\right)}$ and
$a_{j_2}^{\left(t\right)}=\gl^{r_{t}}a_{j_1}^{\left(t\right)}$, so that
\[
c_{j_1}^{\left(t\right)}a_{j_1}^{\left(t\right)}+c_{j_2}^{\left(t\right)}a_{j_2}^{\left(t\right)}
=\left(c_{j_1}^{\left(t\right)}+\gl^{r_{t}}c_{j_2}^{\left(t\right)}\right)a_{j_1}^{\left(t\right)} .
\]
The parameter $\widetilde{w}$ obtained from $w$ by replacing $c_{j_1}$ with
$c_{j_1}+\gl^{r_{t}}c_{j_2}$ and setting both $c_{j_2}$ and the $j_2$-th row of
$w_{t}$ to zero therefore satisfies
$\Phi_X\left(\widetilde{w}\right)=\Phi_X\left(w\right)$, and lies in
$\W\left(Y\right)$ with $Y$ obtained by deleting the $j_2$-th neuron of the
$t$-th layer.

Finally, if the $j$-th column of $w_{t+1}$ vanishes, the $j$-th neuron of the $t$-th layer gives no contribution, and the parameter obtained from $w$ by setting the $j$-th row of $w_{t}$ to zero as well realises the same function and lies in $\W\left(Y\right)$ with $Y$ as above.

In every case $\Phi_X\left(w\right)\in\F\left(Z\right)$ for a proper subnetwork $Z$, so that 
\[
\Phi_X\left(\W\setminus\W\c\right)\su\bigcup_{Z\text{ proper}}\F\left(Z\right)
\]
and we conclude taking closures.

For the last assertion, note that there are finitely many proper subnetworks of $X$, and that each $\F\left(Z\right)$ is irreducible.  If $\F\left(X\right)=\Da\left(X\right)$ then $\F\left(X\right)$ is the finite union of the closed subsets $\F\left(Z\right)$; being irreducible, it coincides with one of them. 
\end{proof}

\begin{cor}\label{cor:NoMassimoStretto}
If no hidden layer of $X$ is a strict maximum, that is if for every
$t=1,\dots,L-1$ either $\dim\N_{t}\le\dim\N_{t-1}$ or
$\dim\N_{t}\le\dim\N_{t+1}$, then
\[
\Da\left(X\right)=\overline{\bigcup_{\d\text{ proper}}\F\left(X^{\d}\right)}.
\]
\end{cor}

\begin{proof}
This follows from Proposition \ref{prop:DarchSottoreti} since deleting the $j$-th neuron from the layer $t$ gives that 
\[
\rk w_t\le \dim \N_t -1\quad\text{and}\quad \rk w_{t+1}\le\dim\N_t -1
\]
and at least one of these inequalities is a proper rank condition.
\end{proof}

\begin{ex}\label{ex:ReteNonRidotta}
A fully connected network can fail to be reduced while having full parameters.
Take
\[
X:\CC^2 \longrightarrow\CC^3 \longrightarrow\CC,\qquad r=2 .
\]
Here $\W\c\neq\e$ by \eqref{eq:CondizioneEsisteParametroPieno}: a general $w_1$
has rank $2$ and pairwise non-proportional rows $\ell_1 ,\ell_2 ,\ell_3$, and a
general $w_2 =\pr{c_1 ,c_2 ,c_3}$ has no vanishing entry.  The realised
functions are the binary quadrics $\sum_{i=1}^{3}c_{i}\ell_{i}^2 $, and every
binary quadric is already a sum of two squares: over $\CC$ it factors as $\ell\ell\p$ and
\[
\ell\ell\p=\frac{1}{4}\left(\pr{\ell+\ell\p}^2 -\pr{\ell-\ell\p}^2 \right).
\]
Hence $\F\pr{Y}=\F\pr{X}$ for the subnetwork $Y$ obtained by deleting one neuron
of the hidden layer, and $X$ is not reduced.
\end{ex}

\subsection{Symmetries}

Set $\ga_{i}=\dim\N_{i}$, following \cite{UBDC25} we let
\[
\G\left(\widehat{X}\right)=\prod_{i=1}^{L-1}\left(\mathbb{G}_{m}^{\ga_{i}}\rtimes S_{\ga_{i}}\right)
\]
and make  it act on $\widehat{\W}$ by
\begin{equation}\label{eq:AzioneGruppo}
g\cdot w=\left(P_{i}D_{i}\,w_{i}\,D_{i-1}^{-r_{i-1}}P_{i-1}^{-1}\right)_{i=1,\dots,L},
\qquad g=\left(D_{i},P_{i}\right)_{i},
\end{equation}
with the convention $D_{0}=P_{0}=D_{L}=P_{L}=\mathrm{id}$, and put
\begin{equation}\label{eq:DefinizioneGruppo}
\G\left(X\right)=\left\{ g\in\G\left(\widehat{X}\right)\,\vert\,g\cdot\W=\W\right\} .
\end{equation}
We call $\G\left(X\right)$ the group of symmetries of the network $X$.

\begin{ex}
Since $\G\left( X\right)$ in \eqref{eq:AzioneGruppo} acts via permutations and product by diagonal matrices, the locus $\W\c$ of full parameters of Definition \ref{def:LuogoDegenere} is $\G \left(X\right)$-invariant.
\end{ex}

It follows from \cite[Lemma 4]{UBDC25} that for every $g\in\G\left(\widehat{X}\right)$ and every $w\in\widehat{\W}$ we have
\begin{equation}\label{eq:InvarianzaPhi}
    \Phi_{\widehat{X}}\left(g\cdot w\right)=\Phi_{\widehat{X}}\left(w\right).
\end{equation}
If moreover $g\in \G\left(X\right)$, then in view of \eqref{eq:InvarianzaPhi} we have $\Phi_X\circ g=\Phi_X$, hence taking differentials, since $g$ is an isomorphism, we conclude that
\begin{equation}\label{eq:InvarianzaDifferenzialePhi}
d\Phi_{X,g\cdot w}\left(T_{\W,g\cdot w}\right)=d\Phi_{X,w}\left(T_{\W,w}\right).
\end{equation}

\begin{defn}\label{def:Rigido}
Let $p\in\Phi_X\left(\W\right)$. We say that $X$ is identifiable at $p$, or that $p$ is rigid, if $\Phi_X^{-1}\left(p\right)$ is a single $\G\left(X\right)$-orbit.
\end{defn}

\subsubsection{The symmetry group of linear convolutional networks}
\label{Sec:LCNCalcoloGruppo}

We compute the symmetry group in a non-fully-connected example. Let $X$ be the transpose of a linear convolutional network as in \cite{KMMT22}. After complexifying, transposing the convolutional matrices and reversing the order of the layers, we identify the layer spaces with
\[
\N_i=\CC\left[x\right]_{\le d_i},\qquad 1\le d_0\le\dots\le d_L,
\]
with fixed bases $1,x,\dots,x^{d_i}$. Given strides $s_1,\dots,s_L\ge1$ and filter degrees $k_1,\dots,k_L\ge0$, the dimensions satisfy $d_i=s_i d_{i-1}+k_i$, and the $i$-th layer associated with a polynomial $f_i\in\CC[x]_{\le k_i}$ is the linear map
\[
\begin{aligned}
A_i\pr{f_i}:\N_{i-1}&\longrightarrow\N_i,\\
p\pr{x}&\longmapsto f_i\pr{x}p\left(x^{s_i}\right).
\end{aligned}
\]
Writing
\[
f_i(x)=\sum_{a=0}^{k_i}f_{i,a}x^a,
\]
the matrix of $A_i\pr{f_i}$ in the monomial bases is given by $\left(f_{i,\,a-s_i b}\right)$ for $0\le a\le d_i$ and $0\le b\le d_{i-1}$, where it is understood that $f_{i,c}=0$ whenever $c\notin\{0,\dots,k_i\}$. This is the transpose of a
convolution matrix of stride $s_i$.
The parameter space is therefore the smooth affine space
\[
\W=\prod_{i=1}^{L}\CC[x]_{\le k_i},
\]
and the realisation associated with
$\left(f_1,\dots,f_L\right)$ is
\[
\Phi_X\left(f_1,\dots,f_L\right)(p)
=
\left(
\prod_{i=1}^{L}
f_i\left(x^{s_{i+1}\cdots s_L}\right)
\right)
p\left(x^{s_1\cdots s_L}\right),
\]
where an empty product of strides is understood to be equal to $1$.

We now determine $\G\pr{X}$. Since all the activation degrees are equal to $1$,
the action \eqref{eq:AzioneGruppo} can be written as
\[
A_i\pr{f_i}\longmapsto g_i A_i\pr{f_i}g_{i-1}^{-1},
\qquad
g_i=P_iD_i,
\]
with $g_0=g_L=\operatorname{id}$. We claim that every $g_i$ is a scalar
multiple of the identity. We prove this by descending induction.
The assertion holds for $g_L=\operatorname{id}$. Suppose that
$g_i=c_i\operatorname{id}_{\N_i}$
for some $i\ge2$. Since $g\in\G\pr{X}$, applying it to the layer corresponding to
the constant filter $1$ gives
\begin{equation}\label{eq:FormulaA}
c_iA_i(1)g_{i-1}^{-1}=A_i(h_i)
\end{equation}
for some $h_i\in\CC[x]_{\le k_i}$. The matrix $g_{i-1}^{-1}$ is has precisely one non-zero element in every row and column, so
there are a permutation $\tau$ of $\{0,\dots,d_{i-1}\}$ and non-zero scalars
$\lambda_0,\dots,\lambda_{d_{i-1}}$ such that $g_{i-1}^{-1}\left(x^j\right) = \lambda_jx^{\tau(j)}$. Evaluating \eqref{eq:FormulaA} at $x^j$ gives
\[
c_i\lambda_jx^{s_i\tau(j)}
=
h_i(x)x^{s_i j}
\]
which, for $j=0$, yields $h_i(x)=c_i\lambda_0 x^{s_i\tau(0)}$. 
Substituting this expression back and comparing exponents and coefficients, we obtain
\[
\tau(j)=j+\tau(0),
\qquad
\lambda_j=\lambda_0
\]
for every $j$. Since $\tau$ is a permutation of the finite interval
$\{0,\dots,d_{i-1}\}$, the first equality forces $\tau(0)=0$ and hence
$\tau=\operatorname{id}$. The second then shows that
$g_{i-1}^{-1}=\gl_0\operatorname{id}_{\N_{i-1}}$ is scalar, and therefore $g_{i-1}=c_{i-1}\operatorname{id}_{\N_{i-1}}$. 
for some $c_{i-1}\in\CC^*$. Descending from $g_L=\operatorname{id}$, we conclude
that
\[
g_i=c_i\operatorname{id}_{\N_i}
\qquad\text{for }i=1,\dots,L-1,
\]
with the conventions $c_0=c_L=1$.
Conversely, every such collection of scalar transformations preserves $\W$.
Indeed,
\[
c_iA_i\pr{f_i}c_{i-1}^{-1} = A_i\left(\frac{c_i}{c_{i-1}}f_i\right).
\]
Setting $u_i=\dfrac{c_i}{c_{i-1}}$, we have that $u_1\cdots u_L = \dfrac{c_1}{c_0}\dfrac{c_2}{c_1}\cdots\dfrac{c_L}{c_{L-1}} = \dfrac{c_L}{c_0} = 1$. It follows that
\[
\G\pr{X} \simeq \left\{ \pr{u_1,\dots,u_L}\in\mathbb G_m^L \;\middle|\; u_1\cdots u_L=1 \right\} \simeq\mathbb G_m^{L-1},
\]
with action
\[
\pr{f_1,\dots,f_L}\longmapsto \pr{u_1f_1,\dots,u_Lf_L}.
\]

\subsection{The quotient}

A remarkable property of the simplicity of the Definition \eqref{eq:DefinizioneGruppo} is that it gives an algebraic group with good categorial properties. We recall that, given a map $f:T\to S$, a subset $U\su T$ is saturated if $f^{-1}\left( f\left(U\right)\right)=U$.

\begin{thm}\label{thm:RiduttivitaLineare}
The group $\G\left(X\right)$ is linearly reductive, consequently $$\Om\left(X\right)=\mathrm{Spec}\left(\CC\left[\W\right]^{\G\left(X\right)}\right)$$ is an irreducible affine variety and the quotient map $\pi_X:\W\to\Om\left(X\right)$ is a good quotient, that is
\begin{enumerate}
    \item it is surjective, it is open on saturated open subsets;
    \item two points have the same image if and only if the closures of their orbits intersect;
    \item each fibre contains exactly one closed orbit;
    \item every morphism $\W\to Z$ constant on the $\G\left(X\right)$-orbits factors uniquely through $\pi_X$.
\end{enumerate}
\end{thm}
 
\begin{proof}

First note that, given $g\in\G\left(\widehat{X}\right)$, from $g\cdot\W\su\W$ it follows that $g\cdot \W=\W$, hence
\[
\G\left(X\right) = \bigcap _{w\in\W} \rho ^{-1} _w \left(\W\right)
\]
where $\rho _w :\G\left(\widehat{X}\right)\to\widehat{\W}$ is given by $g\mapsto g\cdot w$, in particular $\G\left(X\right)$ is a closed subgroup of $\G\left(\widehat{X}\right)$. Moreover $\G\left(\widehat{X}\right)$ is an extension of the finite group $\Gamma=\prod_{i}S_{\ga_{i}}$ by the torus $T=\prod_{i}\mathbb{G}_{m}^{\ga_{i}}$, and a connected subgroup has trivial image in $\Gamma$, so the connected component of the identity $\G\left(X\right)^{\circ}$ is a closed connected subgroup of $T$, that is a subtorus. Since a torus is linearly reductive because it is diagonalizable and a finite group is linearly reductive in characteristic zero by Maschke's theorem, and linear reductivity is stable under extensions \cite[Chapter 1, Section 1]{MuFo82}, we see that $\G\left(X\right)$ is linearly reductive. The remaining assertions are the standard properties of the affine quotient by a linearly reductive group \cite[Chapter 1, Section 2]{MuFo82}; irreducibility holds because $\CC\left[\W\right]^{\G\left(X\right)}$ is a subring of the domain $\CC\left[\W\right]$.
\end{proof}

Since $\Phi_X$ is $\G\left(X\right)$-invariant in view of \eqref{eq:InvarianzaPhi}, it induces a parametrisation $\phi_X:\Om\left(X\right)\longrightarrow\F\left(X\right)$ of the neurovariety.

The following lemma, which is a key technical step in applying Theorem \ref{thm:OmegaLiscio}, essentially states that, after excluding neurons that are not used either as inputs or outputs, the orbits of the action are closed.

\begin{lem}\label{lem:OrbiteChiuse}
Let $w\in\W$ be such that no $w_{i}$ with $i\le L-1$ has a vanishing row and no
$w_{i}$ with $i\ge2$ has a vanishing column. Then the orbit
$\G\left(X\right)\cdot w$ is closed in $\W$.
\end{lem}

\begin{proof}
Let $G=\G\left(X\right)^{\circ}$ be the connected component of the identity in $\G\left(X\right)$. The elements of $\G\left(\widehat{X}\right)$ permute and rescale rows and columns, so every point of $\G\left(X\right)\cdot w$ satisfies the same hypotheses as $w$; since $\G\left(X\right)/G$ is finite, the orbit $\G\left(X\right)\cdot w$ is a finite union of such $G$-orbits, and it is enough to prove that each of them is closed. By the Hilbert-Mumford criterion \cite[Chapter 2]{MuFo82}, were $G\cdot w$ not closed there would be a one-parameter subgroup of $G$ whose limit at $0$ exists and lies outside $G\cdot w$; since $G$ is a subtorus of $T=\prod_{i}\mathbb{G}_{m}^{\ga_{i}}$, it is therefore enough to show that no non-trivial one-parameter subgroup $\gl$ of $T$ admits a limit $\lim_{t\to0}\gl\left(t\right)\cdot w$ in the linear representation $\widehat{\W}$. Write $\gl\left(t\right) = \left(D_{i} \left(t\right) \right)_{i}$ with $D_{i}\left(t\right) = \mathrm{diag}\left(t^{a_{i1}}, \dots ,t^{a_{i\ga_{i}}}\right)$, with the conventions $D_{0}=\mathrm{Id}$, $D_{L}=\mathrm{Id}$ and $r_{0}=1$. By \eqref{eq:AzioneGruppo},
\[
\left(\gl\left(t\right)\cdot w\right)_{i,jk} = t^{\,a_{ij}-r_{i-1}a_{i- 1,k}} \left(w_{i} \right)_{jk},
\]
so the limit exists if and only if
\begin{equation}\label{eq:Pesi}
a_{ij}\ \ge\ r_{i-1}a_{i-1,k}\qquad\mbox{whenever }\left(w_{i}\right)_{jk}\neq0 .
\end{equation}
We show that $a_{ij}\le0$ for every $i,j$ by descending induction on $i$, the case $i=L$ being the convention $a_{Lj}=0$ for every $j$. Let $2\le i\le L$ and assume $a_{ij}\le0$ for every $j$. Every column $k$ of $w_{i}$ is non-zero, so there is $j$ with $\left(w_{i}\right)_{jk}\neq0$, and \eqref{eq:Pesi} gives $r_{i-1}a_{i-1,k}\le a_{ij}\le0$, whence $a_{i-1,k}\le0$. Symmetrically we show that $a_{ij}\ge0$ for every $i,j$ by ascending induction on $i$, the case $i=0$ being the convention $a_{0k}=0$ for every $k$. Let $1\le i\le L-1$ and assume $a_{i-1,k}\ge0$ for every $k$. Every row $j$ of $w_{i}$ is non-zero, so there is $k$ with $\left(w_{i}\right)_{jk}\neq0$, and \eqref{eq:Pesi} gives $a_{ij}\ge r_{i-1}a_{i-1,k}\ge0$. Hence $a_{ij}=0$ for every $i$ and every $j$, that is $\gl$ is trivial, and $G\cdot w$ is closed.
\end{proof}

\begin{rem}\label{rem:DueCasiOrbiteChiuse}
In view of Lemma \ref{lem:OrbiteChiuse} we see that the orbit $\G\left(X\right)\cdot w$ is closed in the following two cases:
\begin{enumerate}
\item $w\in\W\c$, since the rows of $w_{t}$ are non-zero for every hidden layer
$t$ and the outgoing columns are non-zero, both by Definition
\ref{def:LuogoDegenere};
\item all the entries of all the $w_{i}$ are non-zero, which is the case for a general point of $\W$ and, more generally, of $\W^{\d}$ for every $\d$ with $d_{i}\ge1$.
\end{enumerate}
\end{rem}
\begin{rem}\label{rem:AzioneLibera}Let $X$ be an algebraic network.  Then $\G\left(X\right)$ acts freely on $\W\c$: if $\left(P_{i},D_{i}\right)_{i}$ fixes $w\in\W\c$ then, since $P_0 =D_0 =\mathrm{id}$, we get $P_1 D_1 w_1 =w_1$, so every row of $w_1$ is a scalar multiple of another row of $w_1$; the rows being pairwise non-proportional the permutation is trivial, and being non-zero, also$D_1 =\mathrm{id}$.  Inductively $P_{t}=D_{t}=\mathrm{id}$ for $t=1,\dots,L-1$. Together with Remark \ref{rem:DueCasiOrbiteChiuse}~(1) this says that on the full locus the action is free and its orbits are closed for every algebraic network, with no hypothesis on the widths; this is why the condition below asks nothing about them. \end{rem}
From now on we will mostly refer to networks that satisfy the following
\begin{equation}\label{cond:ReteDocile}
  \boxed{
    \parbox{0.8\textwidth}{\centering
      $\W\c\neq\e$ contained in the smooth locus of $\W$.
    }
  }
\end{equation}

\begin{thm}\label{thm:OmegaLiscio}
Let $X$ be an algebraic network that satisfies \eqref{cond:ReteDocile}, then $\W\c$ is saturated and $\Om\c\left(X\right) = \pi_{X}\left(\W\c\right)$ is open in $\Om\left(X\right)$. Moreover
\begin{enumerate}
    \item $\pi_{X}:\W\c\to\Om\c\left(X\right)$ is an \'etale $\G\left(X\right)$-torsor;
    \item $\ker d\pi_{X,w} = \g\left(X\right)\cdot w$ for every $w\in\W\c$.
\end{enumerate} 
In particular $\dim\Om\left(X\right)=\dim\W-\dim\G\left(X\right)$, and
$\Om\c\left(X\right)$ is smooth at $\pi_{X}\left(w\right)$ whenever $\W$ is
smooth at $w$.
\end{thm}

\begin{proof}
Let $w\in\W\c$ and $w\p\in\W$ with $\pi_{X}\left(w\p\right)=\pi_{X}\left(w\right)$.
By Theorem \ref{thm:RiduttivitaLineare} the closures of the two orbits meet, and
the orbit of $w$ is closed in view of Remark \ref{rem:DueCasiOrbiteChiuse}~(1), so $\overline{\G\left(X\right)\cdot w\p}$
meets $\G\left(X\right)\cdot w\su\W\c$.  Now $\G\left(X\right)\cdot w\p$ is a
finite union of orbits under the identity component $\G\left(X\right)^{\circ}$,
hence $\overline{\G\left(X\right)\cdot w\p}$ is the union of their closures, and
one of them, say $\overline{\G\left(X\right)^{\circ}\cdot w^{\prime\prime}}$, meets the open set $\W\c$.  That closure is irreducible with dense subset
$\G\left(X\right)^{\circ}\cdot w^{\prime\prime}$, so the non-empty open subset
$\W\c\cap\overline{\G\left(X\right)^{\circ}\cdot w^{\prime\prime}}$ meets
$\G\left(X\right)^{\circ}\cdot w^{\prime\prime}$; therefore $w^{\prime\prime}\in\W\c$ and, $\W\c$ being
$\G\left(X\right)$-invariant, $w\p\in\W\c$.  We conclude since the image of a
saturated open subset under $\pi_{X}$ is open.
\begin{enumerate}
\item 
The orbits contained in $\W\c$ are closed and their stabilisers are
trivial, in view of Remarks \ref{rem:DueCasiOrbiteChiuse}~(1) and \ref{rem:AzioneLibera}, so by Luna's slice theorem \cite{Lun73} there exists a locally closed subset $w\in S_w \su \W$ together with a cartesian diagram
\[
\xymatrix{
\G\left(X\right) \times S_w \ar[d]  _{\mathrm{pr}}\ar[r] ^{\psi _w} & \W \ar[d] ^{\pi _X} \\
S_w \ar[r]^{\overline\psi _w} & \Om\left( X\right) }
\]
where $\overline\psi _w$ is \'etale with $\overline\psi _w \left( w\right)=\pi _X \left( w\right)$. Denoting with $S\c _w =S_w \times_{\Om \left(X\right)} \Om\c\left(X\right)$ we have an \'etale map $S\c_w\to\Om\c\left(X\right)$ and by cartesianity we conclude that
\[
\G\left(X\right)\times S_w \c\simeq \pi_X ^{-1}\left(\Om\c \left(X\right) \right)\times _{\Om\c\left(X\right)} S\c_w = \W\c\times _{\Om\c\left(X\right)} S\c_w .
\]
In particular we have an \'etale covering $S=\coprod_w S_w\c\to\Om\c\left(X\right)$ and a $\G\left(X\right)$-equivariant isomorphism
\[
\G\left(X\right)\times S \simeq \W\c\times_ {\Om\c\left(X\right)} S
\]
making the restriction $\pi_X :\W\c\to\Om\c\left(X\right)$ an \'etale $\G\left(X\right)$-torsor.
\item A torsor is a smooth morphism, so, setting $y=\pi_{X}\left(w\right)$, the
sequence of tangent spaces
\[
0\to T_w \left(\left( \pi _X\right)^{-1} \left(y\right)\right) \to T_w \W\c
\xrightarrow{d\pi_{X,w}}T_y\Om\c\left(X\right)\to0
\]
is exact; the fibre of a torsor being the orbit, we have $\left( \pi _X\right)^{-1} \left(y\right)=\G \left(X\right)\cdot w$, hence
\[
\ker \left( d\pi_{X,w}\right) = T_w \left(\left( \pi _X\right)^{-1} \left(y\right)\right) = T_w \left(\G \left(X\right)\cdot w\right) =\g\left(X\right)\cdot w.
\]
\end{enumerate}
Finally $\pi_{X}$ is smooth and surjective of relative dimension
$\dim\G\left(X\right)$, and $\dim\W\c=\dim\W$ because $\W\c$ is a non-empty open
subset of the irreducible variety $\W$; hence
$\dim\Om\c\left(X\right)=\dim\W-\dim\G\left(X\right)$, which is also
$\dim\Om\left(X\right)$ since $\Om\c\left(X\right)$ is a non-empty open subset of
the irreducible $\Om\left(X\right)$.  Smoothness descends along a smooth
surjective morphism, so $\Om\c\left(X\right)$ is smooth at $\pi_{X}\left(w\right)$
whenever $\W$ is smooth at $w$.
\end{proof}


\begin{cor}\label{cor:OmLiscioTotalmenteConnesse}
Let $X$ be a fully connected network such that $\W\c\neq\e$, then $X$ satisfies \eqref{cond:ReteDocile} and $\Om\c\left(X\right)$ is a smooth open subset of $\Om\left(X\right)$.
\end{cor}
\begin{proof}
Being fully connected, $\W=\widehat\W$ is a linear space, so $\W\c$, which is open in $\W$, is contained in its smooth locus; together with $\W\c\neq\e$ this is \eqref{cond:ReteDocile}.  The last assertion is then Theorem \ref{thm:OmegaLiscio}.
\end{proof}

\begin{ex}\label{ex:LCNdocili}
Corollary \ref{cor:OmLiscioTotalmenteConnesse} covers every fully connected network with non-empty full locus. The hypotheses of Theorem \ref{thm:OmegaLiscio} are, however, satisfied also by networks which are not fully connected, as the following instance shows.  Consider the network $X$ from (\ref{Sec:LCNCalcoloGruppo}). With that setting and notations, the matrix of multiplication by $f_{i}$ is injective exactly when $f_{i}\neq0$, and its $k$-th column is the coefficient vector of $x^{s_i k}f_{i}$; therefore
\[
\W\c\su\left\{ \left(f_1 ,\dots,f_{L}\right)\,\vert\,f_{i}\neq0\mbox{ for every }i\right\}
\]
and it is smooth provided it is non-empty, being an open subset of a vector space. We see that $\W\c\neq\e$ when $s_1 =\dots=s_{L-1}=1$, indeed for every $i$ consider
\[
f_i \pr{x}=\sum_{a=0}^{k_i}\pr{a+1} x^a\quad\text{for}\quad i<L
\]
and $f_L=1$. The multiplication by $f_i$ is injective, hence $M_{f_i}$ has maximal rank, no column of $M_{f_i}$ vanishes. Note also that $M_{f_i}$ has no vanishing row: the row corresponding to the coefficient of $x^j$ in the image is given by $\pr{f_{i,j},f_{i,j-1},\dots,f_{i,j-d_{i-1}}}$ (with the convention that $f_{i,a}=0$ whenever $a>k_i$) that has at least one non-zero entry and two distinct rows have supports $\left[\max\{0,j-k_i\},\min\{d_{i-1},j\}\right]$ and can therefore be proportional only if both supports are the whole interval $\left[0,d_{i-1}\right]$; in that case the two rows are $\left(j+1,j,\dots,j+1-d_{i-1}\right)$ and $\left(j\p+1,j\p,\dots,j\p+1-d_{i-1}\right)$, whose first and last entries have different ratios unless $j=j\p$.

It follows that $\W\c\neq\e$ and $X$ satisfies \eqref{cond:ReteDocile}, so Theorem \ref{thm:OmegaLiscio} applies to $X$, even if the conclusion is vacuous since $\F\pr{X}$ is a linear space. Unfortunately $\W\c=\e$ whenever $s_t>1$ for some hidden layer $t\le L-1$. Assume
$d_0\ge1$ and $s_t>1$ for some hidden layer $t\le L-1$, then the first two rows
of the matrix of $p\pr{x}\longmapsto f_t\pr{x}p\pr{x^{s_t}}
$ are
\[
\pr{f_{t,0},0,\dots,0}
\qquad\text{and}\qquad
\pr{f_{t,1},0,\dots,0},
\]
hence one of them vanishes or they are proportional, in either case $\W\c=\emptyset$.
\end{ex}

As a first application of the existence of $\Om \left(X\right)$ we see a structural result for fully connected networks.

\begin{defn}\label{def:Identifiability}
We say that $X$ is finitely identifiable if every fibre of $\phi_X$ is finite, and generically finitely identifiable if the general fibre of $\phi_X$ is finite or, equivalently, if
$\dim\Om\left(X\right)=\dim\F\left(X\right)$. 
\end{defn}

\begin{prop}\label{prop:GFIimplicaRidotta}
Let $X$ be a fully connected network which is generically finitely identifiable. Then $X$ is reduced.
\end{prop}

\begin{proof}
Since $\F\left(Y^{\d}\right)\su\F\left(Y\right)$ and $\F\left(Y\p\right)\su\F\left(Y\right)$ whenever $Y\p$ is obtained from $Y$ by deleting further neurons, we may assume that either $Y=X$, or $\d$ is maximal and exactly one neuron has been deleted. Suppose first that no neuron is removed, hence $\F\left(X^{\d}\right)=\F\left(X\right)$ for a proper $\d$.
Note that $\G\left(X^{\d}\right) = \G\left(X\right)$ for every $\d$ since the action of $\G\left(X\right)$ preserves the rank. Let $w$ be a point all of whose entries are non-zero: its orbit is closed in view of Remark \ref{rem:DueCasiOrbiteChiuse} (2), and its stabiliser is finite, since the torus part of a stabilising element satisfies $t_{ij}=t_{i-1,k}^{r_{i-1}}$ for every $j,k$, whence $t_{ij}=1$ by induction from $t_{0}=1$, while the permutation parts form a finite group, in particular the stabiliser is zero-dimensional and hence $\dim\Om\left(X\right)=\dim\W-\dim\G\left(X\right)$. The same applies to a general point of $\W^{\d}$ whenever $d_{i}\ge1$ for every $i$, whose entries are all non-zero. 
Since $X$ realises non-zero functions, then $d_{i}\ge1$ for every $i$ and the general point of the determinantal variety $\W^{\d}$ has rank exactly $d_{i}$ and all entries non-zero, so the same argument gives $\dim\Om\left(X^{\d}\right) = \dim\W^{\d}-\dim\G\left(X\right)$. Since $\d$ is proper we have $\dim\W^{\d}<\dim\W$ and hence $\dim\Om\left(X^{\d}\right) < \dim\Om\left(X\right)$. But $X$ is generically finitely identifiable, so 
\[
\dim\F\left(X^{\d}\right)\le\dim\Om\left(X^{\d}\right) < \dim\Om\left(X\right)=\dim\F\left(X\right)
\]
which is a contradiction.
If $\F\left(X\right)=\F\left(Y\right)$ for a subnetwork $Y$ obtained by removing a neuron in the layer $t$, then
\[
\dim \W\left(Y\right)=\dim\W-\dim\N_{t-1}-\dim\N_{t+1}\quad\text{and}\quad\dim\G\left(Y\right)=\dim\G\left(X\right)-1
\]
hence, in view of the same argument as before,  $\dim\Om\left(Y\right)\le\dim\Om\left(X\right)-1$, but then
\[
\dim \F\left(Y\right)\le\dim\Om\left(Y\right)<\dim\Om\left(X\right)=\dim\F\left(X\right)
\]
which is again a contradiction.
\end{proof}

\begin{ex}
In the case of Example \ref{ex:ReteNonRidotta} the network is filling, so
$\dim\F\left(X\right)=3$, while a parameter with all entries non-zero has closed
orbit and trivial stabiliser, so that the argument in the proof of Proposition
\ref{prop:GFIimplicaRidotta} gives
\[
\dim\Om\left(X\right)=\dim\W-\dim\G\left(X\right)=9-3=6 ,
\]
and hence $X$ is not generically finitely identifiable.
\end{ex}

\subsection{Good parameters}

The symmetries force $\g\left(X\right)\cdot w\su\ker d\Phi_{X,w}$ for every parameter, and those for which this is an equality are useful for two reasons: the condition is pointwise, computable and open, so that a single rank computation certifies global properties of the architecture (Corollary \ref{cor:CriterioRango}), and it is the local notion through which the sources of singularities will be described in Section \ref{Sec:Fonti}.

\begin{prop}\label{prop:BuonoEquivalenze}
Let $X$ be an algebraic network that satisfies \eqref{cond:ReteDocile}, let $\g\left(X\right)=\Lie\G\left(X\right)$ and let $w\in\W\c$. The following are equivalent:
\begin{enumerate}
\item $\ker d\Phi_{X,w}=\g\left(X\right)\cdot w$;
\item $\phi_X$ is unramified at $\left[w\right]=\pi_X\left(w\right)$;
\item $\rk d\Phi_{X,w}=\dim T_{\W,w}-\dim\G\left(X\right)$.
\end{enumerate}
Moreover the set of parameters $w\in\W\c$ that satisfy the equivalent conditions above is open in $\W$ and $\G\left(X\right)$-invariant.
\end{prop}

\begin{proof}
In view of Theorem \ref{thm:OmegaLiscio} the map $\pi_X$ is smooth on $\W\c$ with $\ker d\pi_{X,w}=\g\left(X\right)\cdot w$, and $\Phi_X=\phi_X\circ\pi_X$, so
\begin{equation}\label{eq:UguaglianzaKer}
    \ker d\Phi_{X,w}=\left(d\pi_{X,w}\right)^{-1}\left(\ker d\phi_{X,\left[w\right]}\right).
\end{equation}
Since $d\pi_{X,w}$ is surjective with kernel $\g\left(X\right)\cdot w$, the equality \eqref{eq:UguaglianzaKer} shows that (1) and (2) are equivalent. Since the action on $\W\c$ is free in view of Remark \ref{rem:AzioneLibera}, we have $\dim\left[\g\left(X\right)\cdot w\right]=\dim\G\left(X\right)$ and (1) is equivalent to $\dim\ker d\Phi_{X,w}=\dim\G\left(X\right)$, that is to (3). 

Openness follows since $\rk  d\Phi_{X,w}\ge\dim T_{\W,w}-\dim\G\left(X\right)$ is an open condition on the smooth locus of $\W$ and the reverse inequality always holds. Finally invariance follows from \eqref{eq:InvarianzaDifferenzialePhi}.
\end{proof}

\begin{defn}\label{def:Parametribuoni}
Let $w\in\W\c$, we say that $w$ is a good parameter if the equivalent conditions of Proposition \ref{prop:BuonoEquivalenze} hold.
\end{defn}

\begin{cor}\label{cor:CriterioRango}
Let $X$ be a fully connected network with $\W\c\neq\e$, then
\[
\rk d\Phi_{X,w}\le\sum_{i=1}^{L}\dim\N_{i-1}\dim\N_{i}-\sum_{i=1}^{L-1}\dim\N_{i}
\] 
for every $w\in\W\c$, with equality if and only if $w$ is a good parameter. If such a $w$ exists, then 
\[
\dim\F\left(X\right)=\dim\Om\left(X\right) = \sum_{i=1}^{L}\dim\N_{i-1}\dim\N_{i}-\sum_{i=1}^{L-1}\dim\N_{i},
\]
and $X$ is generically finitely identifiable and reduced.
\end{cor}

\begin{proof}
Set $\nu\left(X\right)=\sum_{i=1}^{L}\dim\N_{i-1}\dim\N_{i}-\sum_{i=1}^{L-1}\dim\N_{i}$. Since $X$ is fully connected we have $\dim\W = \sum_{i=1}^{L}\dim\N_{i-1}\dim\N_{i}$, and $\G\left(X\right)=\prod_{i=1}^{L-1}\left(\mathbb{G}_{m}^{\dim\N_{i}}\rtimes S_{\dim\N_{i}}\right)$ has dimension $\sum_{i=1}^{L-1}\dim\N_{i}$. The locus $\W\c$ is non-empty, a general parameter satisfying all the conditions of Definition \ref{def:LuogoDegenere}, and the action of $\G\left(X\right)$ on it is free in view of Remark \ref{rem:AzioneLibera}. As $\Phi_X$ is $\G\left(X\right)$-invariant, $\g\left(X\right)\cdot w\su\ker d\Phi_{X,w}$ and the left-hand side has dimension $\dim\G\left(X\right)$, which gives the inequality. The case of equality is Proposition \ref{prop:BuonoEquivalenze}.

Let now $w\in\W\c$ be a good parameter. Since $\W$ is a vector space and the ground field has characteristic zero, $\dim\F\left(X\right)$ is the generic value of $\rk d\Phi_{X,v}$, which by the above is at most $\nu\left(X\right)$ and is attained at $w$, moreover in view of Theorem \ref{thm:OmegaLiscio} 
\[
\dim\Om\left(X\right)=\dim\W-\dim\G\left(X\right) = \nu\left(X\right).
\]
Thus $\phi_X$ is a dominant morphism between irreducible varieties of the same dimension, so its general fibre is finite and $X$ is generically finitely identifiable; reducedness then follows from Proposition \ref{prop:GFIimplicaRidotta}.
\end{proof}

For every $w\in\W\c$ one has
\[
\rk d\Phi_{X,w}\le\dim\F\left(X\right)\le\dim\Om\left(X\right).
\]
A parameter $w\in\W\c$ is good precisely when $\rk d\Phi_{X,w}=\dim\Om\left(X\right)$, thus the existence of a single good parameter forces $\dim\F\left(X\right)=\dim\Om\left(X\right)$, that is, generic finite identifiability.
In this sense, Corollary \ref{cor:CriterioRango} has a practical consequence. Being good is an open condition, so the rank of $d\Phi_{X,w}$ at a single $w$ certifies at once that the architecture $X$ is reduced and that its parametrisation has finite generic fibres without elimination nor Gr\"obner basis: it is one rank computation on a matrix of size $\dim\left(X\o\ot\Sym^R \left(X\i\inv\right)\right)\times\dim\W$.

\section{Exceptional singularities}

\begin{defn}\label{def:LuogoEccezionale}
Let $X$ be an algebraic network. The exceptional locus of $X$ is
\[
\E\left(X\right)=\left(\Sing\F\left(X\right)\cap\Phi_X\left(\W\right)\right)\setminus\Da\left(X\right)
\]
and we refer to its elements as exceptional singularities.
\end{defn}

Note that if $p\in\Phi_X\left(\W\right)$ and $p\notin\Da\left(X\right)$, then every $w\in\Phi_X^{-1}\left(p\right)$ lies in $\W\c$, since otherwise $p\in\Phi_X\left(\W\setminus\W\c\right)\su\Da\left(X\right)$. Therefore $\E\left(X\right)=\e$ means precisely that every function which does not lie in the closure of the set of functions admitting a degenerate representation, is a smooth point of the neurovariety.

Two cases in which the exceptional locus is easily seen to be empty are the following.

\begin{ex}\label{ex:FillingEVuoto}
If $X$ is filling, then $\F\left(X\right)$ is a linear space hence it is smooth, so $\E\left(X\right)=\e$.
\end{ex}

\begin{ex}\label{ex:LineareEVuoto}
Let $X$ be fully connected and linear, that is $r_{i}=1$ for every $i$, then
\[
\F\left(X\right)=\left\{ T\in\Hom\left(\N_{0},\N_{L}\right)\,\vert\,\rk T\le d\right\},
\]
where $d=\min_{0\le t\le L}\dim\N_{t}$. If $d=\dim\N_{0}$ or $d=\dim\N_{L}$ this is the whole $\Hom\left(\N_{0},\N_{L}\right)$ and we conclude by Example \ref{ex:FillingEVuoto}. Otherwise $\Sing\F\left(X\right)=\left\{ \rk T\le d-1\right\} $, and every map of rank at most $d-1$ factors through a subspace of dimension $d-1$, hence is realised by a parameter with a matrix of non-maximal rank. In any case $\Sing\F\left(X\right)\su\Da\left(X\right)$ and $\E\left(X\right)=\e$.
\end{ex}

We point out that, in view of Theorem \ref{thm:OmegaLiscio}, over the complement of the architectural degeneracy locus the quotient is smooth, so that every singular point of the neurovariety realised by a full parameter is a defect of the parametrisation map and not of its source. 

\subsection{The sources of exceptional singularities}\label{Sec:Fonti}

We want to investigate now the ways to produce a point of $\E\left(X\right)$. While Definition \ref{def:Fonti} makes sense for arbitrary algebraic networks, in this section we work with networks $X$ that satisfy \eqref{cond:ReteDocile}. The main example we keep in mind is the case of fully connected networks with non-empty full locus (Corollary \ref{cor:OmLiscioTotalmenteConnesse}). We write $\pi_X\colon\W\to\Om\left(X\right)$ for the quotient map and $\phi_X\colon\Om\left(X\right)\to\F\left(X\right)$ for the morphism induced by $\Phi_X$, so that $\Phi_X=\phi_X\circ\pi_X$. In particular $\phi_X\left(\Om\left(X\right)\right)=\Phi_X\left(\W\right)$ and $\phi_X^{-1}\left(p\right)=\pi_X\left(\Phi_X^{-1}\left(p\right)\right)$ for every $p\in\F\left(X\right)$.  In view of Theorem \ref{thm:OmegaLiscio} the image $\Om\c\left(X\right)=\pi_X\left(\W\c\right)$ is a smooth open subset of $\Om\left(X\right)$ and $\pi_X\colon\W\c\to\Om\c\left(X\right)$ is a $\G\left(X\right)$-torsor. 

\begin{defn}\label{def:Fonti}
Let $p\in\F\left(X\right)$. Following the terminology of \cite{SMK26}, we say that
\begin{enumerate}
\item $p$ is a cusp if there exists $w \in \Phi_X^{-1} \left(p\right)$ such that $\rk d\Phi_{X,w} < \dim\F\left(X\right)$. We denote with $C\left(X\right)\su\F\left(X\right)$ the set of cusps.
\item $p$ is a node if $\phi_X^{-1} \left(p\right)$ is not connected. We denote with $N\left(X\right)\su\F\left(X\right)$ the set of nodes.
\end{enumerate}
Finally we denote with $P\left(X\right)\su\F\left(X\right)$ the set of points around which the map $\phi_X$ is not proper. 
\end{defn}

Here we adopt the convention for which the empty set is connected, hence if $p\not\in\Phi\left(\W\right)$ then $p\in P\left(X\right)$.

\begin{rem}
It is shown in \cite[Appendix A]{SMK26} that belonging to $C\left(X\right)$ or $N\left(X\right)$ are two independent conditions, and that a cusp need not be a singular point; the same is true of a node, as the parametrisation $t\mapsto t^2$ of a line shows.
\end{rem}

\begin{note}
We recall that, for a morphism $f:T\to S$ of algebraic varieties over $\CC$, the GAGA correspondence between algebraic and analytic properness applies and algebraic properness of $f$ is equivalent to analytic properness of $f^{\mathrm{an}}: T^{\mathrm{an}} \longrightarrow S^{\mathrm{an}}$. An algebraic morphism is proper if it is separated, of finite type and universally closed and one can see that, if both $T$ and $S$ are affine, algebraic properness is equivalent to $f$ being a finite morphism. On the other hand analytic properness means that for every analytic compact $K\su S^{\mathrm{an}}$ the subset $\left(f^{\mathrm{an}}\right)^{-1}\left( K\right) \su T^{\mathrm{an}}$ is compact. With an abuse of notation, since the set of closed points of an algebraic variety and its analytification is the same, we will omit the superscript $\mathrm{an}$.
\end{note}

Note that, if $X$ is generically finitely identifiable and $p\notin\Da\left(X\right)$, then $p$ is a cusp precisely when some parameter over $p$ is not good.  

\begin{thm}\label{thm:Singolarita}
    Let $X$ be an algebraic network that satisfies \eqref{cond:ReteDocile}, then
    \[
    \Sing \F\left(X\right)\su\Da\left(X\right)\cup C\left(X\right)\cup N\left(X\right) \cup P\left(X\right).
    \]
\end{thm}

\begin{proof}
Let $p\in\F\left(X\right)$ and suppose that 
\[
p\not\in\Da\left(X\right)\cup C\left(X\right)\cup N\left(X\right)\cup P\left(X\right).
\]
In view of Chevalley's theorem there exists a Zariski open subset $O\su\phi_{X}\left(\Om\left(X\right)\right)$ dense in $\F\left(X\right)$; as $\F\left(X\right)$ is irreducible, $O$ is also dense in the analytic topology. Since $p\notin P\left(X\right)$ there is an open neighbourhood $p\in U_{0}\su\F\left(X\right)$ such that the restriction $\phi_{X}^{-1}\left(U_{0}\right)\to U_{0}$ is proper, hence its image is closed in $U_{0}$ and contains $O\cap U_{0}$, which is dense in $U_{0}$; the same argument applies to every open $U\su U_{0}$, so
\begin{equation}\label{eq:PhiSuriettivaAperti}
\phi_{X}\bigl(\phi_{X}^{-1}\left(U\right)\bigr)=U
\qquad\text{for every open }U\su U_{0} .
\end{equation}
In particular $p\in\phi_{X}\left(\Om\left(X\right)\right)=\Phi_{X}\left(\W\right)$.

Denote $K_{p}=\phi_{X}^{-1}\left(p\right)$, it is non-empty by the above, compact by properness, connected since $p\notin N\left(X\right)$ and it is contained in $\Om\c\left(X\right)$ since $p\notin\Da\left(X\right)$.

Since $\W\c$ is dense in $\W$ we have $\overline{\phi_{X}\left(\Om\c\left(X\right)\right)}=\F\left(X\right)$, so on the smooth irreducible variety $\Om\c\left(X\right)$ the rank of $d\phi_{X}$ never exceeds $\dim\F\left(X\right)$ and, being lower semicontinuous, the locus where it attains this value is open. Let $q\in K_{p}$ and let $w\in\W\c$ be such that $\pi_{X}\left(w\right)=q$. Since $\pi_{X}:\W\c\to\Om\c\left(X\right)$ is a smooth torsor we have $d\Phi_{X,w}=d\phi_{X,q}\circ d\pi_{X,w}$ with $d\pi_{X,w}$ surjective, whence
\[
\rk d\phi_{X,q}=\rk d\Phi_{X,w}=\dim\F\left(X\right),
\]
the second equality because $p$ is not a cusp. Therefore $q$ has an open neighbourhood $W_{q}\su\Om\c\left(X\right)$ on which $d\phi_{X}$ has constant rank $\dim\F\left(X\right)$, and after shrinking it the constant rank theorem makes $\phi_{X}\left(W_{q}\right)$ a locally closed complex submanifold of $\F\left(X\right)$ of dimension $\dim\F\left(X\right)$ through $p$, in particular it is smooth.

If $q\p\in W_{q}\cap K_{p}$ we may take $W_{q\p}\su W_{q}$, and then $\phi_{X}\left(W_{q\p}\right)\su\phi_{X}\left(W_{q}\right)$ are two smooth germs at $p$ of the same dimension, hence equal. The germ of $\phi_{X}\left(W_{q}\right)$ at $p$ is thus a locally constant function of $q\in K_{p}$, and $K_{p}$ is connected, so all the $q\in K_{p}$ define one and the same smooth germ $\Gamma$ at $p$.

Let $W_{q_{1}},\dots,W_{q_{n}}$ be a finite subcover of the compact set $K_{p}$ and set $V=W_{q_{1}}\cup\dots\cup W_{q_{n}}$. Only finitely many submanifolds are involved, and each of them represents $\Gamma$, so there are a compactly contained neighbourhood $p\in U\su U_{0}$ of $p$ and a representative $G$ of $\Gamma$, with
\[
\phi_{X}\left(V\right)\cap U = G\cap U.
\]
By properness we can also suppose that $\phi_{X}^{-1}\left(U\right)\su V$, hence
\[
U=\phi_{X}\bigl(\phi_{X}^{-1}\left(U\right)\bigr)
\su\phi_{X}\left(V\right)\cap U=G\cap U\su U
\]
therefore $U=G\cap U$, that is, $\F\left(X\right)$ coincides near $p$ with the
complex manifold $G$ of dimension $\dim\F\left(X\right)$, and $p$ is a smooth
point of $\F\left(X\right)$.
\end{proof}

\begin{cor}\label{cor:Fonti}
    Let $X$ be an algebraic network that satisfies \eqref{cond:ReteDocile}, then
    \[
    \E\left(X\right)\su C\left(X\right)\cup N\left(X\right) \cup P\left(X\right).
    \]
\end{cor}

The following lemma gives a tool we will use to produce points of $\E\left(X\right)$. Note that, in view of \eqref{eq:InvarianzaDifferenzialePhi}, the subspaces $T_{j}$ below depend only on the $\G\left(X\right)$-orbit of $v_{j}$. 

\begin{lem}\label{lem:CriterioTangenziale}
Let $X$ be an algebraic network, let $v_1 ,\dots,v_{k}\in\W$ be smooth points with $\Phi_X\left(v_1 \right) = \dots = \Phi_X\left(v_{k}\right)=:p$ and suppose that $p\notin \Da\left(X\right)$. Set $T_{j} = d\Phi_{X,v_{j}} \left(T_{\W,v_{j}}\right)$. If 
\[
\dim\left(T_1 +\dots+T_{k}\right)>\dim\F\left(X\right),
\]
then $p\in\E\left(X\right)$; in particular $p$ is not rigid.
\end{lem}

\begin{proof}
The map $\Phi_X$ factors through the closed subscheme $\F\left(X\right)$ of the ambient space, so $T_{j}\su T_{p}\F\left(X\right)$ for every $j$ and therefore under our assumptions
\[
\dim T_{p}\F\left(X\right)\ \ge\ \dim\left(T_1 +\dots+T_{k}\right)\ >\ \dim\F\left(X\right),
\]
that is $p\in\Sing\F\left(X\right)$. Since $p\in\Phi_X\left(\W\right)$ and $p\notin\Da\left(X\right)$ we get $p\in\E\left(X\right)$. Finally, suppose $p$ were rigid. Then the $v_{j}$ would lie in a single $\G\left(X\right)$-orbit and all the $T_{j}$ would coincide by \eqref{eq:InvarianzaDifferenzialePhi}, so that $\dim\left(T_1 +\dots+T_{k}\right)=\dim T_1 $. But $\rk d\Phi_{X,v}$ is lower semicontinuous in $v$, so its generic value is the maximum one, namely $\dim\F\left(X\right)$; hence $\dim T_1 \le\dim\F\left(X\right)$, contradicting the hypothesis.
\end{proof}

\begin{rem}\label{rem:Fonti}
It follows from Corollary \ref{cor:Fonti} that every point of $\E\left(X\right)$
exhibits at least one of the three behaviours described in Definition
\ref{def:Fonti}.  Cuspidality is treated in \cite[Theorem 5.1]{KMSSW26}, where it
is shown that for a monomial network of activation degree $r>9\left(D\left(D-1\right)\right)^{L-1}$, where
$D$ is the largest hidden width, criticality occurs exactly at subnetworks,
that is at parameters having an inactive or a redundant neuron; such parameters are not full by Definition \ref{def:LuogoDegenere}, so that these points lie
inside $\Da\left(X\right)$.
A second one is self-intersection: distinct orbits of full parameters realise the same function, and their tangent images are distinct.  It is this second occurrence that Lemma
\ref{lem:CriterioTangenziale} is about, and we will see in Section
\ref{Sec:Esempio} an example of how this can be used to produce elements of
$\E\left(X\right)$.  The third one, namely lack of properness, appears in Section
\ref{Sub:Propria} and in Remark \ref{rem:PerchePuntuale}.
\end{rem}

\subsection{The case of a proper parametrisation}\label{Sub:Propria}

Both $\Om\left(X\right)$ and $\F\left(X\right)$ are affine, so $\phi_X$ is proper if and only if it is finite. Finiteness clearly implies the finite identifiability, but the other direction isn't true, a quasi-finite morphism being in general not finite. Proposition \ref{prop:CriterioFinitezza} will show nevertheless that for a cone the two conditions coincide. 
Assume from now on that $\W$ is a cone, which is the case for all the examples of Example \ref{ex:Esempi}. The group $\mathbb G _m$ acts on $\W$ by $w\mapsto\left(tw_1 ,\dots,tw_{L}\right)$, commuting with $\G\left(X\right)$, and $\Phi_X\left(t\cdot w\right)=t^g \Phi_X\left(w\right)$ where
\[
g=1+r_{L-1}\left(1+r_{L-2}\left(1+\dots\right)\right)>0 .
\]
Note that the ideal of $\F\left(X\right)$ is homogeneous: if $f$ vanishes on $\F\left(X\right)$ and $f=\sum_{d}f_{d}$ is its decomposition into homogeneous components, then for every $w\in\W$ and every $t$
\[
0=f\left(\Phi_X\left(t\cdot w\right)\right)=f\left(t^g \Phi_X\left(w\right)\right)=\sum_{d}t^{gd}f_{d}\left(\Phi_X\left(w\right)\right),
\]
and the exponents $gd$ being distinct, each $f_{d}$ vanishes on $\Phi_X\left(\W\right)$, hence on $\F\left(X\right)$. Writing $A=\CC\left[\Om\left(X\right)\right]$, the rings $\CC\left[\W\right]$, $A$ and $\CC\left[\F\left(X\right)\right]$ are positively graded with degree-zero part $\CC$. We see that the only point of $\Om\left(X\right)$ that is fixed by $\mathbb G _m$ is $\pi_X\left(0\right)$, and to do this we first claim that a point $q\in\Om\left(X\right)$ is fixed by $\mathbb G _m$ if and only if it kills $A_+$. Indeed suppose that $q=t\cdot q$ for every $t\in\CC\inv$ and let $a$ be homogeneous of degree $d>0$, then $a\left(q\right)=t^d a\left(q\right)$ for every $t$ hence $a\left(q\right)=0$. On the other hand if $a_d\left(q\right)=0$ for every $a_d$ homogeneous of degree $d>0$, then $a\left(q\right)=a\left(t\cdot q\right)$ for every $a\in A$ hence $q=t\cdot q$ for every $t\in\CC\inv$. Therefore the ideal of fixed points is $A_+$ and it is contained in $\mathfrak{m}_{\pi_X (0)}$ the ideal of the point $\pi_X (0)$ and it is in fact equal to it being both maximal. If $q$ is another fixed point then necessarily $\mathfrak{m}_{\pi_X (0)}=\mathfrak{m}_q$. 

Finally $\Phi_X=\phi_X\circ\pi_X$ gives $\phi_X\left(t\cdot q\right)=t^g \phi_X\left(q\right)$, so $\phi_X$ is graded once $\F\left(X\right)$ is endowed with the action $t\star p=t^g p$.

\begin{prop}\label{prop:CriterioFinitezza}
Let $X$ be an algebraic network such that $\W$ is a cone. The following are equivalent:
\begin{enumerate}
\item $\phi_X$ is finite;
\item $X$ is finitely identifiable;
\item $\phi_X^{-1}\left(0\right)=\left\{ \pi_X \left(0\right) \right\}$;
\item every $w\in\W$ with $\Phi_X\left(w\right)=0$ is unstable for the action of $\G\left(X\right)$, that is 
\[ 0\in\overline{\G\left(X\right)\cdot w} ;\]
\item for every $w\in\W$ with $\Phi_X\left(w\right)=0$ there is a one-parameter subgroup $\gl$ of $\G\left(X\right)$ with \[\lim_{t\to0}\gl\left(t\right)\cdot w=0.\]
\end{enumerate}
\end{prop}

\begin{proof}
A finite morphism is quasi-finite, so (1) implies (2).
 
Assume (2). The fibre $\phi_X^{-1}\left(0\right)$ is $\mathbb{G}_{m}$-stable, being the preimage of a fixed point under a graded morphism; a point of it other than $\pi_X\left(0\right)$ is not fixed, hence has a one-dimensional orbit contained in the fibre, contradicting (2). This shows (3).
 
Assume (3), and write $B=\CC\left[\F\left(X\right)\right]$ and
$A=\CC\left[\Om\left(X\right)\right]$, so that $\phi_X^{\#}:B\to A$ is a morphism of positively graded finitely generated $\CC$-algebras with $A_{0}=B_{0}=\CC$. By (3) the ideal $B_{+}A$ has radical $A_{+}$, so $A/B_{+}A$ is a finitely generated $\CC$-algebra of dimension zero, hence a finite-dimensional $\CC$-vector space. Choose homogeneous $a_1 ,\dots,a_{n}\in A$ whose classes span it, with $a_1 =1$, and set $M=\sum_{i}Ba_{i}$. We prove $A_{d}\su M$ by induction on $d$, the case $d=0$ being $A_{0}=\CC\cdot a_1\su M$. If $a\in A_{d}$ with $d>0$, then $a-\sum_{i}c_{i}a_{i}\in B_{+}A$ for some $c_{i}\in\CC$, so $a=\sum_{i}c_{i}a_{i}+\sum_{j}b_{j}a\p_{j}$ with $b_{j}\in B_{+}$ and $a\p_{j}\in A$ homogeneous of degree smaller than $d$ since $\deg b_j>0$. By induction $a\p_{j}\in M$, whence $a\in M$. Therefore $A=M$ is a finite $B$-module and $\phi_X$ is finite.
 
Note that $\pi_X$ is surjective and $\phi_X\circ\pi_X=\Phi_X$, moreover the orbit of $0$ is clearly closed, hence in view of Theorem \ref{thm:RiduttivitaLineare} we have that $0\in\overline{\G\left(X\right)\cdot w}$ if and only if $\pi_X\left(w\right) = \pi_X\left(0\right)$. The equivalence between (3) and (4) follows.

Finally the equivalence of (4) and (5) is the Hilbert-Mumford criterion \cite[Chapter 2]{MuFo82}.
\end{proof}

Let $X$ be an algebraic network that satisfies \eqref{cond:ReteDocile}, if $p\in\F\left(X\right) \setminus \Da\left(X\right)$, every parameter realising $p$ is full and its orbit is closed in view of Remark \ref{rem:DueCasiOrbiteChiuse}~(1); the fibre $\phi_X^{-1}\left(p\right)$ is therefore exactly the set of $\G\left(X\right)$-orbits realising $p$, so over such a $p$ finite identifiability is literally the finiteness of that set. What we call generic finite identifiability is called finite identifiability in \cite{UBDC25}.

\begin{thm}\label{thm:PuntualeProprio}
Let $X$ be an algebraic network that satisfies \eqref{cond:ReteDocile}, and let $p\in\Phi_X\left(\W\right)\setminus\Da\left(X\right)$. Suppose that
\begin{enumerate}
\item $p$ is rigid, that is $\Phi_X^{-1}\left(p\right)$ is a single $\G\left(X\right)$-orbit (Definition \ref{def:Rigido}), say $\left[w\right]=\pi_X\left(w\right)$;
\item all the parameters over $p$ are good;
\item $\OO_{\Om\left(X\right),\left[w\right]}$ is a finite $\OO_{\F\left(X\right),p}$-module.
\end{enumerate}
Then $p$ is a smooth point of $\F\left(X\right)$. In particular every point of $\E\left(X\right)$ either fails to be rigid, or is realised by a parameter which is not good, or is a point at which (3) fails. 
\end{thm}

\begin{proof}
Since $p\not\in\Da\left(X\right)$, by (1) the fibre $\phi_X^{-1}\left(p\right)$ is the single point $\left[w\right]\in\Om\c\left(X\right)$, and in view of (2) and Proposition \ref{prop:BuonoEquivalenze} the morphism $\phi_X$ is unramified there. In view of (3)  $\OO_{\Om\left(X\right),\left[w\right]}$ is a finite $\OO_{\F\left(X\right),p}$-module, and being unramified
\[
\OO_{\Om\left(X\right),\left[w\right]}\big/\mathfrak{m}_{p}\OO_{\Om\left(X\right),\left[w\right]}=\CC .
\]
By Nakayama the map $\OO_{\F\left(X\right),p} \to \OO_{\Om\left(X\right),\left[w\right]}$ is surjective; it is injective because $\phi_X$ is dominant and $\F\left(X\right)$ is integral. Hence it is an isomorphism, and $\OO_{\F\left(X\right),p} \simeq \OO_{\Om\c\left(X\right),\left[w\right]}$ is regular in view of Theorem \ref{thm:OmegaLiscio}.
\end{proof}

We point out that condition (3) holds at every $p$ as soon as $\phi_X$ is finite, hence under any of the equivalent conditions of Proposition \ref{prop:CriterioFinitezza}.

\begin{note} 
Theorem \ref{thm:Singolarita} is analytic and local on $\F\left(X\right)$: it assumes neither identifiability nor $\dim\F\left(X\right)=\dim\Om\left(X\right)$, and it concludes only that $p$ is a smooth point.  On the other hand in Theorem \ref{thm:PuntualeProprio} we need the stronger condition (1), which forces the fibre to be a single point, and in exchange it identifies the two local rings, so that $\phi_{X}$ is a local isomorphism at $\left[w\right]$; moreover it replaces properness around $p$ by a condition on a single local ring, which is what Proposition \ref{prop:CriterioFinitezza} makes verifiable in practice.
\end{note}


\begin{rem}\label{rem:PerchePuntuale}
Properness is essential in the argument of Theorem \ref{thm:PuntualeProprio}, since it rules out branches arising from parameters escaping to infinity as described in \cite[Remark 2]{BCCGO18}. Let $\ell$ and $m$ be linearly independent linear forms, then for every $t\neq0$, the polynomial
\[
f_t=t^{-1}\left(\left(\ell+tm\right)^r-\ell^r\right)
\]
is realised by a shallow network with two hidden neurons, whereas
\[
f_t = r\ell^{r-1}m  + \binom r2t\ell^{r-2} m^2 +\cdots+t^{r-1}m^r \longrightarrow r\ell^{r-1} m 
\]
for $t\to0$. The corresponding parameters do not extend to $t=0$, since their output coefficients are $t^{-1}$ and $-t^{-1}$, while the realised functions do. For $r>2$, the limit is not a sum of two $r$-th powers, even though it lies in the closure of the locus of such sums. Geometrically, this expresses the fact that tangent lines to the Veronese variety are limits of secant lines. Thus a curve in $\F\left(X\right)$ may extend through a limiting point even though its lift to the parameter quotient does not. Such a branch is not detected by the fibre of $\phi_X$ over the limiting point or by the differentials at the points of that fibre; this is precisely the behaviour excluded by the valuative criterion for properness.
\end{rem}

Unfortunately the hypothesis about the finiteness of $\phi_X$ for fully connected networks is satisfied only in trivial cases.

\begin{prop}\label{prop:FCnonPropria}
Let $X$ be a fully connected network with $L\ge2$ and $\dim\N_{L-1}\ge2$. Then the fibre of $\phi_X$ over the zero function is infinite, in particular $X$ is not finitely identifiable.
\end{prop}

\begin{proof}
In view of Proposition \ref{prop:CriterioFinitezza} it is enough to exhibit a semistable $w$ with $\Phi_X\left(w\right)=0$. 
Choose $w_1,\dots,w_{L-1}$ with all entries non-zero and require the first two rows of $w_{L-1}$ to be equal. Recall that the $j$-th component $q_j^{\left(L-1\right)}$ of the last hidden pre-activation is obtained by pairing the $j$-th row of $w_{L-1}$ with the vector of activated outputs of the preceding layer. Hence the equality of the first two rows gives $q_1^{\left(L-1\right)}=q_2^{\left(L-1\right)}$.
Now choose a non-zero vector $c\in\N_L$ and set
\[
w_L = \begin{bmatrix} c&-c&0&\cdots&0 \end{bmatrix},
\]
thus explicitly
\[
\Phi_X(w)
=w_L
\begin{pmatrix}
\left(q_1^{\left(L-1\right)}\right)^{r_{L-1}}\\
\left(q_2^{\left(L-1\right)}\right)^{r_{L-1}}\\
\vdots\\
\left(q_{\dim\N_{L-1}}^{\left(L-1\right)}\right)^{r_{L-1}}
\end{pmatrix}
=c\left(q_1^{\left(L-1\right)}\right)^{r_{L-1}}
-c\left(q_2^{\left(L-1\right)}\right)^{r_{L-1}}
=0.
\]
In view of Hilbert-Mumford criterion \cite[Chapter 2]{MuFo82}, suppose a one-parameter subgroup of weights $\left(a_{ij}\right)$ satisfies $\lim_{t\to0}\gl\left(t\right)\cdot w=0$; then $a_{ij}-r_{i-1}a_{i-1,k}>0$ for every non-zero entry  $\left(w_{i}\right)_{jk}$. Since $w_1 ,\dots,w_{L-1}$ have all entries non-zero and $a_{0}=0$, an ascending induction gives $a_{ij}>0$ for every $i\le L-1$ and every $j$. On the other hand the entries $\left(w_{L}\right)_{j1}$ with $j$ in the support of $c$ are non-zero, so $0=a_{Lj}>r_{L-1}a_{L-1,1}$ and $a_{L-1,1}<0$, a contradiction. Hence $w$ is semistable, so $\pi_X\left(w\right)$ is a point of $\phi_X^{-1}\left(0\right)$ distinct from $\pi_X\left(0\right)$; being not fixed by $\mathbb G _m$, its orbit is one dimensional and contained in $\phi^{-1}_X\left(0\right)$, which is therefore infinite.
\end{proof}

\begin{rem}\label{rem:DicotomiaArchitetture}
Proposition \ref{prop:FCnonPropria} says that a fully connected network is never
finitely identifiable as soon as it has at least two neurons in its last hidden
layer: already the zero function is realised by infinitely many orbits of
parameters.  The obstruction is the cancellation locus, consisting of
parameters which realise the zero function without being unstable.  It arises
from the freedom of choosing the last layer so as to cancel two equal features
against each other.  The point of Definition \ref{def:ReteAlgebrica}, which
allows an arbitrary parameter space $\W$, is that a constrained architecture
may leave no room for such a cancellation.

Linear convolutional networks give the opposite behaviour.  Consider the
network $X$ of Section \ref{Sec:LCNCalcoloGruppo}, with arbitrary strides.
Its parameter space is a product of vector spaces, hence a cone, and its
realisation map is
\[
\Phi_X(f_1,\dots,f_L)(p)
=
\left(
\prod_{i=1}^{L}
f_i\left(x^{s_{i+1}\cdots s_L}\right)
\right)
p\left(x^{s_1\cdots s_L}\right).
\]
Since the polynomial ring is a domain, the realised map vanishes if and only if
at least one of the filters vanishes.  Therefore
\[
\Phi_X^{-1}(0)
=
\bigcup_{i=1}^{L}\{f_i=0\}.
\]

Let $(f_1,\dots,f_L)$ belong to this fibre and choose an index $i$ such that $f_i=0$.  Choose integers $b_j>0$ for $j\neq i$ and let $b_i$ be such that $b_1+\dots+b_L=0$. The calculation of
Section \ref{Sec:LCNCalcoloGruppo} shows that $u_j(t)=t^{b_j}$ defines a one-parameter subgroup of $\G(X)$, moreover
\[
\lim_{t\to0} \left(t^{b_1}f_1,\dots,t^{b_L}f_L\right) = 0:
\]
the $i$-th component is identically zero, while every other component tends to zero because its exponent is positive.  Thus every parameter realising the zero function is unstable.  Proposition \ref{prop:CriterioFinitezza} therefore
shows that $\phi_X:\Om(X)\longrightarrow\F(X)$ is finite, and hence proper, for every choice of the strides. In particular $P(X)=\emptyset$.

The behaviour of the full locus is more restrictive.  As observed in Example \ref{ex:LCNdocili},
\[
\W\c\neq\emptyset
\quad\Longleftrightarrow\quad
s_1=\dots=s_{L-1}=1;
\]
the last stride $s_L$ is unrestricted.  If some hidden stride is larger than one, then $\W\c=\emptyset$, whence $\Da(X)=\F(X)$ and $\E(X)=\emptyset$. If instead $s_1=\dots=s_{L-1}=1$, then $\W\c\neq\e$ and $\W$ is a product of vector spaces, so $X$ satisfies \eqref{cond:ReteDocile}. Corollary \ref{cor:Fonti} and the properness proved above then give
\[
\E(X)\su C(X)\cup N(X).
\]
When also $s_L=1$, however the conclusion is vacuous for a different reason: over $\CC$, every polynomial of degree at most $\sum_i k_i$ can be factored into polynomials of degrees at most $k_1,\dots,k_L$, so $\F(X)$ is a linear space and is therefore smooth.  Allowing $s_L>1$ gives genuinely non-linear
neurovarieties while preserving the non-emptiness of the full locus. 
\end{rem}

\section{A low dimensional example of an exceptional singularity}\label{Sec:Esempio}

Throughout this section
\[
Y:\CC^2 \longrightarrow\CC^2 \longrightarrow\CC^2 \longrightarrow\CC^2 
\]
is the fully connected network with activation degrees $\rr=\left(2,2\right)$. 
This network appears in \cite[Example 4.6]{KLW24}, where the ideal of
$\F\left(Y\right)$ is presented by the $5\times5$ minors of a $5\times6$ matrix
under the expectation that $\dim\F\left(Y\right)=8$, which will follow from our computation.
We fix coordinates $x,y$ on $Y\i\inv$ and take the canonical basis of $Y\o$, and consider the point $p$ for which two inequivalent full representations were observed in \cite[Example 41]{UBDC25}: set
\[
u=x^2 +y^2 ,\qquad v=x^2 -y^2 ,\qquad p=\left(u^2 ,v^2 \right)\in\left(\Sym^4 \left(\CC^2 \right)\inv\right)^2 .
\]
We determine that its whole fibre consists of three orbits, and we prove that $p$ is a singular point of the neurovariety lying outside the architectural degeneracy locus. We shall prove that $p\in\E\left(Y\right)$. Note at once that
\begin{equation}\label{eq:StimaDimensione}
\dim\F\left(Y\right)\ \le\ \dim\Om\left(Y\right)=12-4=8.
\end{equation}
For a parameter $w=\left(w_1 ,w_2 ,w_3 \right)$ we write 
\[
w_1=\begin{pmatrix}
    \ell_1 \\\ell_2
\end{pmatrix}
\qquad 
q=\left(q_1, q_2\right)=w_2\cdot\begin{pmatrix}
    \ell_1 ^2 \\\ell_2 ^2
\end{pmatrix}
\quad\text{and}\quad\Phi_Y \pr{w}=w_3\cdot\begin{pmatrix}
    q_1 ^2 \\q_2 ^2
\end{pmatrix}.
\]

\begin{lem}\label{lem:TreQuadratiNelFascio}
Let $a=x_1 ^{N}$ and $b=x_2 ^{N}$ be powers of two distinct variables, and set $u=a^2 +b^2 $, $v=a^2 -b^2 $. Then, up to scalars, the linear span $\left\langle u^2 ,v^2 \right\rangle $ contains exactly three squares, namely
\[
u^2 , v^2  \quad\text{and}\quad u^2 -v^2 =\left(2ab\right)^2 ,
\]
and each of the spaces $\left\langle u,v\right\rangle$, $\left\langle u,2ab\right\rangle$, $\left\langle v,2ab\right\rangle$ is spanned by the squares of two elements of
$\left\langle a,b\right\rangle$, determined up to scalars and permutation.
\end{lem}

\begin{proof}
For $\ga\gb\neq0$ the binary form $\ga a^2 +\gb b^2 =\ga x_1 ^{2N}+\gb x_2 ^{2N}$ does not vanish at $x_2 =0$, and dehomogenising at $X=x_1 /x_2 $ its roots are the $2N$-th roots of $-\gb/\ga$, which are pairwise distinct. Two such forms with distinct ratios $\ga:\gb$ therefore have disjoint root sets, hence are coprime; and a form with $2N$ distinct roots is not a square, every root of a square having even multiplicity.
 
Write $su^2 +tv^2 =\left(\sqrt{s}\,u+i\sqrt{t}\,v\right)\left(\sqrt{s}\,u-i\sqrt{t}\,v\right)$. If $st=0$ we get $u^2 $ or $v^2 $, if $st\neq0$ the two factors have the form
\[
\sqrt{s}\,u\pm i\sqrt{t}\,v=\left(\sqrt{s}\pm i\sqrt{t}\right)a^2 +\left(\sqrt{s}\mp i\sqrt{t}\right)b^2 ,
\]
with distinct ratios, hence coprime; so the product is a square only if each factor is, which by the previous paragraph happens only when one of the two coefficients vanishes, that is only for $s=-t$, and then  $su^2 +tv^2 =s\left(2ab\right)^2 $. For the last assertion, the squares of $\left\langle a,b\right\rangle $ form a smooth conic in $\PP\left\langle a^2 ,ab,b^2 \right\rangle =\PP^2 $, and each of the three pencils is a line meeting it in the two indicated points.
\end{proof}

We can compute the fibre of $\Phi _Y$ over $p$. Set 
\[
w=\left(\begin{pmatrix}1&0\\0&1\end{pmatrix},\ \begin{pmatrix}1&1\\1&-1\end{pmatrix},\ \begin{pmatrix}1&0\\0&1\end{pmatrix}\right),
\qquad
w\p=\left(\begin{pmatrix}1&1\\1&-1\end{pmatrix},\ \frac{1}{2}\begin{pmatrix}1&1\\1&-1\end{pmatrix},\ \begin{pmatrix}1&0\\1&-1\end{pmatrix}\right),
\]
\[
w^{\prime \prime}=\left(\begin{pmatrix}1&i\\1&-i\end{pmatrix},\ \frac{1}{4}\begin{pmatrix}2&2\\[2pt]-i&i\end{pmatrix},\ \begin{pmatrix}1&4\\1&0\end{pmatrix}\right),
\]
with 
\[
\left(\ell_1 ,\ell_2 ;q_1 ,q_2 \right)=\left(x,y;u,v\right),\quad
\left(x+y,x-y;u,2xy\right),\quad \left(x+iy,x-iy;v,xy\right).
\]
A direct substitution gives $\Phi_{Y}\left(w\right)=\Phi_{Y}\left(w\p\right)=\Phi_{Y}\left(w^{\prime \prime}\right)=p$, and the nine matrices above are invertible, so the three parameters lie in $\W\c$.

Conversely let $\Phi_{Y}\left(w\right)=p$. If $w_3 $ were not invertible the two components of $p$ would be proportional, which they are not; hence $\left\langle q_1 ^2 ,q_2 ^2 \right\rangle =\left\langle u^2 ,v^2 \right\rangle $ and, in view of Lemma \ref{lem:TreQuadratiNelFascio} applied with $a=x$, $b=y$, the pair $\left\{ q_1 ,q_2 \right\} $ is, up to scalars, one of $\left\{ u,v\right\} $, $\left\{ u,2xy\right\} $, $\left\{ v,2xy\right\} $. In particular $q_1 ,q_2 $ are linearly independent, so $w_2 $ is invertible and
\[
\left\langle \ell_1 ^2 ,\ell_2 ^2 \right\rangle =\left\langle q_1 ,q_2 \right\rangle ;
\]
the last part of Lemma \ref{lem:TreQuadratiNelFascio} then determines $\left\{ \left[\ell_1 \right],\left[\ell_2 \right]\right\} $, hence $w_1 $ up to the action of $\mathbb{G}_{m}^2 \rtimes S_2 $, and $w_2 ,w_3 $ are determined in turn. Each of the three cases occurs and yields one orbit. The three orbits are distinct because $\G\left(Y\right)$ rescales and permutes the $\ell_{i}$, so it preserves the unordered pair
\[
\left\{ \left[\ell_1 \right],\left[\ell_2 \right]\right\} \su\PP^{1},
\]
and the three pairs $\left\{ \left[x\right],\left[y\right]\right\} $, $\left\{ \left[x+y\right],\left[x-y\right]\right\}$, $\left\{ \left[x+iy\right],\left[x-iy\right]\right\}$ are different. We have already proved that $w_3$ and $w_2$ are invertible, while the last part of Lemma \ref{lem:TreQuadratiNelFascio} shows that $w_1$ is invertible as well. Hence every matrix has maximal rank and the matrices being invertible of size $2\times 2$, the rows of $w_1$ and $w_2$ are non-zero and non-proportional. It follows that the whole fibre is contained in $\W\c$ and that $p$ is not rigid in the sense of Definition \ref{def:Rigido}. We are ready to show that $p\in\E\left(Y\right)$, hence in particular that $\E\left(Y\right)\neq\e$


We keep the notation as before: $u=x^2 +y^2 $, $v=x^2 -y^2 $, $p=\left(u^2 ,v^2 \right)$, and $w,w\p\in\W\c$ are the two full parameters realising $p$, in particular
\[
\ell=\left(x,y\right),\ q=\left(u,v\right)\quad\text{and} \quad  \ell\p=\left(x+y,x-y\right),\ q\p=\left(u,2xy\right).
\]
We compute $T_w=d\Phi_{Y,w}\left(T_{\W,w}\right)$ by varying one layer at a time. At $w$ we have
\[
w_1=w_3=\mathrm{Id}_2,\qquad w_2= \begin{pmatrix} 1&1\\1&-1\end{pmatrix},
\]
so that $\Phi_Y(w)=(u^2,v^2)$. Writing
\[
\delta w_3=\begin{pmatrix} a&b\\ c&d \end{pmatrix},
\]
we obtain
\[
\delta\Phi_Y = \delta w_3
\begin{pmatrix}
u^2\\
v^2
\end{pmatrix} = 
\begin{pmatrix}
a u^2+b v^2\\
c u^2+d v^2
\end{pmatrix},
\]
hence varying $w_3$ produces the subspace $\left\langle u^2,v^2\right\rangle \oplus \left\langle u^2,v^2\right\rangle$ independently in the two components. Next vary $w_2$, keeping $w_1$ and $w_3$ fixed. Since $\ell=(x,y)$, the corresponding variation of $q$ is
\[
\delta q = \delta w_2 \begin{pmatrix} x^2\\ y^2 \end{pmatrix}.
\]
The two rows of $\delta w_2$ can be chosen independently, so $\delta q_1$ and $\delta q_2$ are arbitrary and independent elements of $\left\langle x^2,y^2\right\rangle$. Differentiating $\Phi_Y=\left(q_1^2,q_2^2\right)$ at $q=\left(u,v\right)$ gives
\[
\delta\Phi_Y = \left(2u\,\delta q_1,2v\,\delta q_2\right).
\]
Consequently, varying $w_2$ gives $u\left\langle x^2,y^2\right\rangle \oplus v\left\langle x^2,y^2\right\rangle$. It remains to consider the variation of $w_1$. Write $\delta\ell_1=a_1x+b_1y$ and $\delta\ell_2=a_2x+b_2y$, since at $w$ we have $q_1=\ell_1^2+\ell_2^2$ and $q_2=\ell_1^2-\ell_2^2$ it follows that
\[
\begin{aligned}
\delta q_1 &= 2\ell_1\delta\ell_1+2\ell_2\delta\ell_2 = 2a_1x^2+2\left(b_1+a_2\right)xy+2b_2y^2, \\[2mm]
\delta q_2 &= 2\ell_1\delta\ell_1-2\ell_2\delta\ell_2 = 2a_1x^2+2\left(b_1-a_2\right)xy-2b_2y^2.
\end{aligned}
\]
Note that the linear map $\left(b_1,a_2\right) \mapsto \left(b_1+a_2,b_1-a_2\right)$ is invertible, hence the coefficients of $xy$ in $\delta q_1$ and $\delta q_2$ can be chosen independently. The components in $\left\langle x^2,y^2\right\rangle$ have already been obtained independently by varying $w_2$, and therefore the joint variations of $w_1$ and $w_2$ allow $\delta q_1$ and $\delta q_2$ to vary independently in the whole space $\Sym^2\left(\CC^2 \right)\inv$. Their contribution to $T_w$ is consequently
\[
u\Sym^2\left(\CC^2 \right)\inv
\oplus
v\Sym^2\left(\CC^2 \right)\inv.
\]
Combining this with the contribution coming from $w_3$, we obtain
\begin{align}\label{eq:TangenteW}
T_w &= \left(\left\langle u^2,v^2\right\rangle + u\Sym^2\left(\CC^2 \right)\inv
\right) \oplus \left(\left\langle u^2,v^2\right\rangle + v\Sym^2\left(\CC^2 \right)\inv\right)\notag \\
&= \left(\left\langle v^2\right\rangle + u\Sym^2\left(\CC^2 \right)\inv \right)
\oplus \left( \left\langle u^2\right\rangle + v\Sym^2\left(\CC^2 \right)\inv\right)
\end{align}
since $u^2\in u\Sym^2\left(\CC^2 \right)\inv$ and $v^2\in v\Sym^2\left(\CC^2 \right)\inv$. Multiplication by $u$ is injective, so $u\Sym^2 $ has dimension $3$; and $v^2 \notin u\Sym^2 $, because $u=\left(x+iy\right)\left(x-iy\right)$ and $v^2 =\left(x-y\right)^2 \left(x+y\right)^2 $ have no common factor, hence $\left\langle v^2\right\rangle + u\Sym^2\left(\CC^2 \right)\inv$ has dimension 4. Similarly one sees that $\left\langle u^2\right\rangle + v\Sym^2\left(\CC^2 \right)\inv$ has also dimension 4, so that
\begin{equation}\label{eq:RangoW}
\rk d\Phi_{Y,w}=\dim T_{w}=8 .
\end{equation}

We show that the branch coming from $w\p$ has a different tangent space by exhibiting a linear form $\Lambda$ on $\left(\Sym^4 \left(\CC^2 \right)\inv\right)^2 $ with $\Lambda\vert_{T_w}=0$ but $\Lambda\left(\tau\right)\neq0$ for some $\tau\in T_{w\p}$. Let
\[
\Lambda\left(f_1 ,f_2 \right)=\Big(\left[x^3 y\right]+\left[xy^3 \right]\Big)f_2 ,
\]
where $\left[m\right]f$ denotes the coefficient of the monomial $m$ in $f$. By \eqref{eq:TangenteW} it is enough to evaluate $\Lambda$ on $u^2 $ and on $v\Sym^2 $: the form $u^2 $ involves only $x^4 ,x^2 y^2 ,y^4 $, and
\[
v\,x^2 =x^4 -x^2 y^2 ,\qquad v\,y^2 =x^2 y^2 -y^4 ,\qquad v\,xy=x^3 y-xy^3 ,
\]
so $\Lambda$ vanishes on each of them. On the branch through $w\p$, take instead the variation $\delta w_2 \p=\frac{1}{2}\begin{psmallmatrix}0&0\\1&1\end{psmallmatrix}$; since $\left(\ell\p_1 \right)^2 +\left(\ell\p_2 \right)^2 =2u$, it is the
variation with $\delta q_1 =0$ and $\delta q_2 =u$, whose tangent vector is
\[
\tau=w_3 \p \begin{pmatrix} 2q\p_1 \delta q_1  \\ 2q\p_2 \delta q_2  \end{pmatrix} 
= w_3 \p\begin{pmatrix} 0\\4xyu \end{pmatrix}
=\begin{pmatrix}0 \\ -4x^3 y-4xy^3 \end{pmatrix}\in T_{w\p}.
\]
Altogether
\begin{equation}\label{eq:SeparaRami}
T_{w}\su\ker\Lambda,\qquad\Lambda\left(\tau\right)=-8\neq0.
\end{equation}
In order to show that $p\notin\Da\left(Y\right)$, we prove that it is not a
limit of functions realised by non-full parameters. Let $Z\su\left(\Sym^4\left(\CC^2\right)\inv\right)^2$ be the closed subset of pairs of proportional forms, and let $z\in\W\setminus\W\c$. Since all the matrices of $Y$ are $2\times2$, a vanishing row or two proportional rows of $z_1$ or $z_2$ force the corresponding matrix to have rank at most $1$; similarly, a vanishing outgoing column of $z_2$ or $z_3$ forces that matrix to have rank at most $1$. Consequently it is enough to consider rank drops.

If $z_1$ has rank at most $1$, then $\ell_1,\ell_2$ are proportional and the output has the form $c\ot L^4$ for some linear form $L$. If $z_2$ has rank at most $1$, then $q_1,q_2$ are proportional and the output has the form $c\ot Q^2$ for some quadratic form $Q$. Finally, if $z_3$ has rank at most $1$, the output lies in a line of $Y\o$. In every case the two components of $\Phi_Y\left(z\right)$ are proportional, so  $\Phi_{Y}\left(\W\setminus\W\c\right)\su Z$ and, $Z$ being closed, we have
\begin{equation}\label{eq:DarchInZ}
\Da\left(Y\right)\su Z .
\end{equation}
Since $\Phi_{Y}$ factors through $\F\left(Y\right)$, both $T_{w}$ and $\tau$ lie in $T_{p}\F\left(Y\right)$, and $\tau\notin T_{w}$ by \eqref{eq:SeparaRami}. Hence, by \eqref{eq:RangoW} and \eqref{eq:StimaDimensione},
\[
\dim T_{p}\F\left(Y\right)\ \ge\ \dim\left(T_{w}+\left\langle \tau\right\rangle \right)=9\ >\ 8\ \ge\ \dim\F\left(Y\right),
\]
so $p$ is a singular point of $\F\left(Y\right)$. It lies in $\Phi_{Y}\left(\W\c\right)$ and it does not lie in $\Da\left(Y\right)$, since $u^2 $ and $v^2 $ are not proportional and \eqref{eq:DarchInZ} holds. Therefore $p\in\E\left(Y\right)$. For the remaining assertions, $Y$ is fully connected and $w\in\W\c$, in particular $\W\c\neq\e$ so
\[
\rk d\Phi_{Y,w}=8=\dim\Om\left(Y\right)
\]
in view of \eqref{eq:RangoW}; so $w$ is a good parameter and Corollary \ref{cor:CriterioRango} therefore yields $\dim\F(Y)=8$ and shows that $Y$ is generically finitely identifiable and reduced.

\begin{rem}\label{rem:NonRidondante}
It follows from the above computation that the singularity is not an artefact of a redundant architecture: the network $Y$ is reduced and its parametrisation has finite general fibres. Moreover the fibre over $p$ consists of three orbits of full parameters, and the same computation as \eqref{eq:RangoW} give
\[
\rk d\Phi_{Y,w\p} = \rk d\Phi_{Y,w^{\prime\prime}} = 8.
\]
Thus all three parameters are good, so $\phi_Y$ is unramified at every point of its fibre over $p$, while that fibre consists of three distinct points. Hence the singularity at $p$ is a node.
\end{rem}

\begin{rem}\label{rem:SoloGradoDue}
The mechanism is genuinely quadratic in view of Fermat's last theorem for polynomials. Let $A,B$ be coprime non-constant binary forms of the same degree and suppose $sA^r +tB^r =Q^r $ with $st\neq0$. Being over $\CC$ we may assume $s=t=1$. Then $A,B,Q$ are pairwise coprime, so $\left[A:B:Q\right]$ has no base point and defines a non-constant morphism $\PP^{1}\to F_{r}$, where $F_{r}\su\PP^2 $ is the Fermat curve $X^r +Y^r -Z^r =0$. But $F_{r}$ is smooth of genus $\binom{r-1}{2}$, which is positive for $r\ge3$, while a curve admitting a non-constant morphism from $\PP^{1}$ is rational. Hence for $r\ge3$ the only $r$-th powers in the pencil $\left\langle A^r ,B^r \right\rangle $ are $A^r $ and $B^r $, so the same computation as that of the fibre $\Phi_Y^{-1}\pr{p}$ in the case of activation degrees $\left(2,r\right)$ or $\left(r,r\right)$ produces a single orbit and the self-intersection mechanism exhibited above does not arise.
\end{rem}

\section{A generalisation}\label{Sec:Generalizzazione}

We want to generalise the example of Section \ref{Sec:Esempio} to fully connected networks that have two consecutive quadratic activations. For every such network we will construct, unconditionally, three full parameters lying in three distinct orbits and realising the same function outside $\Da$. Checking whether this function is a singular point of the neurovariety then reduces to a single inequality involving the span of the three tangent images and $\dim\Om\left(X\right)$, which we verify for a number of architectures in Remark \ref{rem:CasiVerificati} and expect to hold in general. The hidden widths are asked to be constant, the input layer alone being allowed to be larger.

\begin{defn}\label{def:BloccoQuadratico}
Let $X$ be a fully connected network with widths $m_{0},\dots,m_{L}$. A quadratic block of $X$ is an index $1\le t\le L-2$ such that
\[
r_{t}=r_{t+1}=2,\qquad m_{t}=m_{t+1}=m\ge2 .
\]
\end{defn}

\begin{construction}\label{constr:TreRamiGenerale}

Let $X$ be a fully connected network with $m_s =m$ for every $s\ge1$ and
$m_0 \ge m$, and let $t$ be a quadratic block of $X$. Let $x_1 ,\dots,x_{m_0}$
be coordinates on $X\i$, denote by $\rho:\CC^{m_0}\longrightarrow\CC^{m}$ the
projection onto the first $m$ coordinates and set
\[
\rho_{t}=
\begin{cases}
\rho,&t=1\\
\mathrm{id}_{\CC^{m}},&t>1 .
\end{cases}
\]
Before the block, if $t>1$ take $w_1 =\rho$ and $w_{s}=\mathrm{id}_{\CC^{m}}$
for $2\le s<t$, the same for each of the three parameters; if $t=1$ there is no
parameter preceding the block. In either case the functions entering the block
are
\[
a_1 ,\dots,a_{m},\qquad\text{with}\qquad
a_{j}=x_{j}^{R_{t-1}},\qquad R_{t-1}=r_1 \cdots r_{t-1},
\]
with the convention $R_0 =1$, and they are algebraically independent. Inside the
block set
\[
u=a_1 ^2 +a_2 ^2 ,\qquad v=a_1 ^2 -a_2 ^2 .
\]

The three parameters differ only on the first two neurons of each layer of the
block. Writing $A\oplus\mathrm{id}$ for $A\oplus\mathrm{id}_{\CC^{m-2}}$ and
\[
H=\begin{pmatrix}1&1\\1&-1\end{pmatrix},\qquad
K=\begin{pmatrix}1&i\\1&-i\end{pmatrix},
\]
we define $w,w\p,w^{\prime\prime}$ on the block by
\[
\begin{aligned}
w_{t}&=\rho_{t},
&w_{t+1}&=H\oplus\mathrm{id},
&w_{t+2}&=\mathrm{id}_{\CC^{m}},\\[2pt]
w_{t}\p&=\left(H\oplus\mathrm{id}\right)\circ\rho_{t},
&w_{t+1}\p&=\tfrac{1}{2}H\oplus\mathrm{id},
&w_{t+2}\p&=\begin{pmatrix}1&0\\1&-1\end{pmatrix}\oplus\mathrm{id},\\[2pt]
w_{t}^{\prime\prime}&=\left(K\oplus\mathrm{id}\right)\circ\rho_{t},
&w_{t+1}^{\prime\prime}&=\tfrac{1}{4}\begin{pmatrix}2&2\\-i&i\end{pmatrix}\oplus\mathrm{id},
&w_{t+2}^{\prime\prime}&=\begin{pmatrix}1&4\\1&0\end{pmatrix}\oplus\mathrm{id} .
\end{aligned}
\]
After the block, for $s>t+2$ let $w_{s}$ be any invertible matrix, the same for
each of the three parameters.
One readily checks that after the first quadratic activation we have
\[
w\to \begin{pmatrix} a_1 ^2\\a_2 ^2\end{pmatrix}
\qquad
w\p\to \begin{pmatrix} \left(a_1+a_2\right) ^2\\\left(a_1 -a_2\right)^2\end{pmatrix}
\qquad
w^{\prime\prime}\to \begin{pmatrix} \left(a_1+ia_2\right) ^2\\\left(a_1 -ia_2\right)^2\end{pmatrix},
\]
that after the second quadratic activation we have
\[
w\to \begin{pmatrix} a_1^2 + a_2^2\\a_1^2 - a_2^2\end{pmatrix} = \begin{pmatrix} u\\v\end{pmatrix}
\qquad
w\p\to \begin{pmatrix} a_1^2 + a_2^2\\ 2a_1 a_2\end{pmatrix} = \begin{pmatrix} u\\2a_1 a_2\end{pmatrix}
\qquad
w^{\prime\prime}\to \begin{pmatrix} a_1^2 - a_2^2\\ a_1 a_2\end{pmatrix} = \begin{pmatrix} v\\ a_1 a_2\end{pmatrix},
\]
and that at the end of the block we have
\[
w\to \begin{pmatrix} u^2\\v^2\end{pmatrix}
\qquad
w\p\to \begin{pmatrix} u^2\\ u^2-4a_1 ^2 a_2 ^2\end{pmatrix} = \begin{pmatrix} u^2\\v^2\end{pmatrix}
\qquad
w^{\prime\prime}\to \begin{pmatrix} v^2 +4a_1^2 a_2^2\\ v^2\end{pmatrix}  = \begin{pmatrix} u^2\\v^2\end{pmatrix},
\]
hence in particular
\[
p=\Phi_X\left(w\right)=\Phi_X\left(w\p\right)=\Phi_X\left(w^{\prime \prime}\right).
\]
Note that $w,w\p,w^{\prime\prime}\in\W\c$ and lie in three distinct $\G\left(X\right)$-orbits since the three parameters have different unordered pairs in $\PP\langle a_1,a_2\rangle$, and $\G\left(X\right)$ only rescales and permutes neurons.
\end{construction}

\begin{lem}\label{lem:pFuoriDaDarch}
With the notation of Construction \ref{constr:TreRamiGenerale}, $p\notin\Da\left(X\right)$.
\end{lem}

\begin{proof}
Identify $X\o\ot\Sym^R \left(X\i\inv\right)$ with $\left(\Sym^R \left(X\i\inv\right)\right)^{m}$ by means of the chosen basis, and let
\[
Z=\left\{ \left(f_1 ,\dots,f_{m}\right)\in\left(\Sym^R \left(X\i\inv\right)\right)^{m}\,\vert\,f_1 ,\dots,f_{m}\mbox{ are algebraically dependent}\right\}.
\]
We see that $Z$ is closed: in characteristic zero algebraic independence is detected by the Jacobian criterion \cite[Chapter 16]{Ei95}, so $Z$ is the locus where the Jacobian matrix of $\left(f_1 ,\dots,f_{m}\right)$ with respect to the coordinates of $X\i$ has rank smaller than $m$, and is therefore closed. We claim  that $\Phi_X\left(\W\setminus\W\c\right)\su Z$ but $p\not\in Z$.

Let $w\notin\W\c$ and let $s$ be an index at which $w$ fails to be full. Since the following layer has dimension $m$, it follows that $\rk w_s\le m-1$, hence the functions $q ^{\left(s\right)}=w_s a ^{\left(s-1\right)}$ span a space $V$ with $\dim V\le m-1$ and every component of $\Phi_X\left(w\right)$ is a polynomial in a basis of $V$ that lies in the subalgebra $\CC\left[V\right]$, whose transcendence degree over $\CC$ is at most $m-1$. It follows that the $m$ components are algebraically dependent.

We see now that $p\not\in Z$. The forms $a_1 ,\dots,a_{m}$ are algebraically independent, hence so are $a_1 ^2 ,\dots,a_{m}^2 $. The pair $\left(u,v\right)$ is obtained from $\left(a_1 ^2 ,a_2 ^2 \right)$ by an invertible linear change, so $u,v,a_3 ^2 ,\dots,a_{m}^2 $ are algebraically independent, and $u^2 ,v^2 ,a_3 ^4 ,\dots,a_{m}^4 $ as well, which are the components of $p$ inside the block. If there are layers after the block, the components of $p$ are obtained from these by applying invertible matrices and coordinatewise powers, both of which operations preserve algebraic independence, hence $p\notin Z$.
\end{proof}

While the Construction \ref{constr:TreRamiGenerale} always produces a point $p\not\in\Da$ in view of Lemma \ref{lem:pFuoriDaDarch}, we could not prove that it is always the case that $p\in\E\left(X\right)$, and the following result is the criterion that we will use in Remark \ref{rem:CasiVerificati}.

\begin{thm}\label{thm:Generalizzazione}
Let $X$ and $t$ be as in Construction \ref{constr:TreRamiGenerale} and let $p$ be the point constructed there. If the three tangent images at $w,w\p,w^{\prime \prime}$ span a space of dimension larger than $\dim\Om\left(X\right)$, then $p\in\E\left(X\right)$, in particular $\E\left(X\right)\neq\e$.
\end{thm}

\begin{proof}
The three parameters lie in $\W\c$ with the same image $p$, and from Lemma \ref{lem:pFuoriDaDarch} we have $p\notin\Da\left(X\right)$. 
Since $\F\left(X\right)$ is the closure of the image of $\Om\left(X\right)$ we have that $\dim\F\left(X\right) \le \dim\Om\left(X\right)$, and the hypothesis gives $\dim T_{p}\F\left(X\right)>\dim\F\left(X\right)$. We conclude in view of Lemma \ref{lem:CriterioTangenziale}.
\end{proof}

\begin{rem}\label{rem:CasiVerificati}
We verify computationally the hypothesis of Theorem \ref{thm:Generalizzazione} in the following cases. We write $\gb=\dim\Om\left(X\right)$ for the bound and $\tau$ for the dimension of the span of the three tangent images; in every row each of the three tangent images has, separately, dimension exactly $\gb$.
\begin{center}
\begin{tabular}{llrrr}
\hline 
widths & $\rr$ & $\gb$ & $\tau$ & $\tau-\gb$\tabularnewline
\hline 
$\left(2,2,2,2\right)$ & $\left(2,2\right)$ & $8$ & $10$ & $2$\tabularnewline
$\left(3,2,2,2\right)$ & $\left(2,2\right)$ & $10$ & $12$ & $2$\tabularnewline
$\left(4,2,2,2\right)$ & $\left(2,2\right)$ & $12$ & $14$ & $2$\tabularnewline
$\left(5,2,2,2\right)$ & $\left(2,2\right)$ & $14$ & $16$ & $2$\tabularnewline
\hline 
$\left(2,2,2,2,2\right)$ & $\left(2,2,2\right)$ & $10$ & $12$ & $2$\tabularnewline
$\left(2,2,2,2,2,2\right)$ & $\left(2,2,2,2\right)$ & $12$ & $14$ & $2$\tabularnewline
$\left(2,2,2,2,2,2,2\right)$ & $\left(2,2,2,2,2\right)$ & $14$ & $16$ & $2$\tabularnewline
$\left(2,2,2,2,2\right)$ & $\left(2,2,3\right)$ & $10$ & $12$ & $2$\tabularnewline
$\left(2,2,2,2,2\right)$ & $\left(3,2,2\right)$ & $10$ & $12$ & $2$\tabularnewline
\hline 
$\left(3,3,3,3\right)$ & $\left(2,2\right)$ & $21$ & $28$ & $7$\tabularnewline
$\left(4,3,3,3\right)$ & $\left(2,2\right)$ & $24$ & $31$ & $7$\tabularnewline
$\left(3,3,3,3,3\right)$ & $\left(2,2,2\right)$ & $27$ & $34$ & $7$\tabularnewline
$\left(4,4,4,4\right)$ & $\left(2,2\right)$ & $40$ & $52$ & $12$\tabularnewline
$\left(5,4,4,4\right)$ & $\left(2,2\right)$ & $44$ & $56$ & $12$\tabularnewline
$\left(4,4,4,4,4\right)$ & $\left(2,2,2\right)$ & $52$ & $64$ & $12$\tabularnewline
$\left(4,4,4,4,4\right)$ & $\left(2,2,3\right)$ & $52$ & $64$ & $12$\tabularnewline
$\left(5,5,5,5\right)$ & $\left(2,2\right)$ & $65$ & $82$ & $17$\tabularnewline
$\left(6,6,6,6\right)$ & $\left(2,2\right)$ & $96$ & $118$ & $22$\tabularnewline
$\left(7,7,7,7\right)$ & $\left(2,2\right)$ & $133$ & $160$ & $27$\tabularnewline
\hline 
\end{tabular}
\par\end{center}
In every case $\tau>\gb$, so each of these architectures has a non-empty exceptional locus. The excess $\tau-\gb$ depends only on the width $m$ of the quadratic block, and equals $5m-8$ throughout the range $2\le m\le7$: it is insensitive to the dimension of the input, to the depth, and to the activation degrees outside the block. Note in particular the last two rows of the second block: the statement is not one about quadratic networks, but about networks containing a quadratic block. 

The single ranks carry two further consequences. Since each of the three tangent images already has dimension $\gb$, the three parameters are good and Corollary \ref{cor:CriterioRango} applies: every network in the table is generically finitely identifiable and reduced. Moreover $\rk d\Phi_{X,v}\le\dim\F\left(X\right)\le\gb$ forces $\dim\F\left(X\right)=\gb$, so none of these architectures is defective; and since each tangent image is contained in $T_{\F\left(X\right),p}$, the table bounds the singularity from below as well,
\[
\dim T_{\F\left(X\right),p}-\dim\F\left(X\right)\ \ge\ \tau-\gb=5m-8 .
\]
The exceptional points produced here are therefore not mildly singular: the
tangent space exceeds the neurovariety by an amount growing with the width.

Appending a layer $\N_{L}\to\N_{L+1}$ increases $\dim\Om\left(X\right)$ by
$\dim\N_{L}\left(\dim\N_{L+1}-1\right)$, the new layer contributing
$\dim\N_{L}\dim\N_{L+1}$ parameters and $\dim\N_{L}$ further symmetries, and in
each verified case the span of the tangent images grows by exactly the same
amount, leaving $\tau-\gb$ unchanged; whether this persists for arbitrary depth
is part of Problem \ref{prob:Eccesso}.

The verification is a finite exact computation over $\mathbb{Q}\left(i\right)$,
realised as $\mathbb{Q}\left[t\right]/\left(t^2 +1\right)$, so that the square
root of $-1$ of Construction \ref{constr:TreRamiGenerale} is the class of $t$;
the matrices $w_{s}$ with $s>t+2$ are taken to be the identity. For each of the
three parameters we differentiate $\Phi_X$ with respect to every entry of every
$w_{s}$, propagating the chain rule through the layers so that all intermediate
quantities remain polynomials in $x_1 ,\dots,x_{m_{0}}$, and we read off the
dimension of each tangent image and the dimension $\tau$ of their span as ranks
of the coefficient matrices in the monomial basis of
$X\o\ot\Sym^R \left(X\i\inv\right)$, the bound $\gb$ being given by the formula
of Corollary \ref{cor:CriterioRango}. The computation was carried out in
Macaulay2 \cite{M2}. Since all the entries involved lie in $\mathbb{Q}\left(i\right)$
and the rank of a matrix is unchanged under field extension, the values in the
table are the ranks over $\CC$.
\end{rem}

\section{The case of shallow networks: positive results and a counterexample}\label{Sec:Shallow}

In this section we apply the preceding machinery to the case of shallow fully connected networks. There are several reasons for which this is a good place to start: without doubt the presence of a single intermediate layer makes the description of the network more explicit and the theory computationally easier, but the main help comes from the fact that the (projective) variety associated to the network is a well-known classical object: the secant variety of a Segre-Veronese embedding. With this geometric picture in mind we can rely on a number of results from their projective geometry, so the first part of this section is essentially a translation of known facts on secant varieties into the language of Sections \ref{Sec:Setup} and \ref{Sec:Fonti}. We show, nevertheless, that $\E\left(X\right)$ need not be empty even in the shallow case: Example \ref{ex:ShallowEccezionale} exhibits a shallow network with $\E\left(X\right)\neq\e$, so the phenomenon of Section \ref{Sec:Esempio} is not a prerogative of depth.

Throughout, $X:X\i\to\N\to X\o$ is a fully connected shallow network of activation degree $r$, we set $d=\dim\N$, and for $w=\pr{w_1 ,w_2}$ we denote by $\ell_1 ,\dots,\ell_{d}$ the rows of $w_1 $ and by $c_1 ,\dots,c_{d}$ the columns of $w_2 $, so that
\[
\Phi_X\left(w\right)=\sum_{i=1}^d c_{i}\ot\ell_{i}^r 
\]
and $\F\left(X\right)$ is the affine cone over $\gs_{d}\left(\PP\left(X\o\right)\times v_{r}\PP\left(X\i\inv\right)\right)$, as in Example \ref{ex:ShallowSecante}. If the widths are non-increasing, that is 
\begin{equation}\label{eq:NonCrescenteShallow}
\dim X\i\ \ge\ d\ \ge\ \dim X\o,
\end{equation}
then $\W\c$ is the set of parameters with $\ell_1 ,\dots,\ell_{d}$ linearly independent, $c_1 ,\dots,c_{d}$ that span $X\o$ and such that $c_{i}\neq0$ for every $i$, in particular the pairwise non-proportionality is automatic.  




\subsection{Injectivity and good parameters}\label{Sec:IniettivitAParametriBuoni}

We look now at parametrisations in the shallow case under the assumption \eqref{eq:NonCrescenteShallow} and show in Proposition \ref{prop:ShallowMonomorfismo} that the map $\phi$ is particularly well behaved on the full locus: it is injective (Proposition \ref{prop:ShallowIniettivo}) and unramified (Proposition \ref{prop:ShallowBuono}), hence a monomorphism. Thus there are no cusps nor nodes that can produce exceptional singularities in this range; the only possible obstruction comes from non-properness.

\begin{prop}\label{prop:ShallowIniettivo}
Suppose that one of the following holds:
\begin{enumerate}
\item \eqref{eq:NonCrescenteShallow} and $r\ge3$;
\item \eqref{eq:NonCrescenteShallow}, $r=2$ and $d=\dim X\o$;
\item $r\ge2d-1$.
\end{enumerate}
Then $\phi_X$ is injective on $\Om\c\left(X\right)$.
\end{prop}

\begin{proof}
Suppose $\sum_{i}c_{i}\ot\ell_{i}^r =\sum_{j}c_{j}\p\ot m_{j}^r $ with both parameters in $\W\c$.

Let first $r\ge3$. Choose $\xi\in X\o\inv$ with $\xi\left(c_{i}\right)\neq0$ and $\xi\left(c_{j}\p\right)\neq0$ for all $i,j$, which is possible since the vectors are finitely many and non-zero, and set $U=\left\langle \ell_1 ,\dots,\ell_{d}\right\rangle $. We have
\[
F:=\sum_{i}\xi\left(c_{i}\right)\ell_{i}^r =\sum_{j}\xi\left(c_{j}\p\right)m_{j}^r \ \in\ \Sym^r U .
\]
For every $\gb\in\left(X\i\inv\right)\inv$ vanishing on $U$, we get $\partial_{\gb} F = r\sum_{j}\xi\left(c_{j}\p\right)\gb\left(m_{j}\right)m_{j}^{r-1} = 0$, but the $m_{j}^{r-1}$ are linearly independent since the $m_{j}$ are, so $\gb\left(m_{j}\right)=0$ for every $j$ and hence $m_{j}\in U$. Both $\left\{ \ell_{i}\right\} $ and $\left\{ m_{j}\right\} $ are then bases of $U$, and letting $\beta$ vary we have
\[
\left\langle \ell_1 ^{r-1},\dots,\ell_{d}^{r-1}\right\rangle =\left\langle \partial_{\gb}F\,\vert\,\gb\in U\inv\right\rangle =\left\langle m_1 ^{r-1},\dots,m_{d}^{r-1}\right\rangle .
\]
Since $r-1\ge2$, a power $m^{r-1}$ lying in the span of the $\ell_{i}^{r-1}$ has all its mixed monomials vanishing, hence $m$ is proportional to some $\ell_{i}$. After a permutation $m_{i}=\gl_{i}\ell_{i}$, and substituting in the original equality gives $c_{i}=\gl_{i}^r c_{i}\p$ by independence of the $\ell_{i}^r $. Thus the two parameters differ by an element of $\G\left(X\right)$.

Let now $r=2$ and $d=\dim X\o$, so that $c_1 ,\dots,c_{d}$ and $c_1 \p,\dots,c_{d}\p$ are bases of $X\o$. Viewing the tensor as a linear map $X\o\inv\to\Sym^2 \left(X\i\inv\right)$ we get
\[
\left\langle \ell_1 ^2 ,\dots,\ell_{d}^2 \right\rangle =\left\langle m_1 ^2 ,\dots,m_{d}^2 \right\rangle .
\]
Write $m_{j}=\sum_{i}\gl_{i}\ell_{i}+n$ with $n$ in a complement of $U$; expanding $m_{j}^2 $ forces $n=0$, and then the vanishing of the mixed terms $2\gl_{i}\gl_{k}\ell_{i}\ell_{k}$ forces at most one $\gl_{i}$ to be non-zero. So again $m_{i}=\gl_{i}\ell_{i}$ up to a permutation and $c_{i}=\gl_{i}^2 c_{i}\p$.

Assume finally that $r\ge2d-1$, and let $w,w\p\in\W\c$ such that $\Phi_X(w)=\Phi_X(w\p)$. Let $L_1,\dots,L_s$ be representatives of the distinct proportionality classes occurring among $\ell_1,\dots,\ell_d$ and $m_1,\dots,m_d$, in particular $s\le2d$. The $\ell_i$ being pairwise non-proportional, each class contains at most one of them, and likewise at most one $m_j$. We have
\[
\sum_{a=1}^{s}\gamma_a\ot L_a^r=0 \quad\text{where}\quad
\gamma_a=
\sum_{\ell_i=\lambda_iL_a}\lambda_i^r c_i
-
\sum_{m_j=\mu_jL_a}\mu_j^r c_j\p.
\]
Since $r\ge2d-1\ge s-1$, in view of \cite[Theorem A.2]{MM25} the powers $L_1^r,\dots,L_s^r$ are linearly independent, hence $\gamma_a=0$ for every $a$.  A class containing an $\ell_i$ but no $m_j$ would give $\lambda_i^r c_i=0$, contradicting $c_i\neq0$ and symmetrically, no class can contain an $m_j$ without containing an $\ell_i$. Therefore, after reindexing, we can suppose that $m_i=\nu_i\ell_i$ and $c_i=\nu_i^r c_i\p$ for some $\nu_i\neq0$, so $w$ and $w\p$ lie in the same $\G\left(X\right)$-orbit.
\end{proof}

\begin{prop}\label{prop:ShallowBuono}
Assume \eqref{eq:NonCrescenteShallow} and either $r\ge3$, or $r=2$ and
$d=\dim X\o$. Then every $w\in\W\c$ is a good parameter.
\end{prop}

\begin{proof}
By Proposition \ref{prop:BuonoEquivalenze} we must show that the kernel of
$d\Phi_{X,w}$ is exactly the tangent space to the orbit. A tangent vector is
\[
d\Phi_{X,w}\left(\delta w\right)=\sum_{i=1}^d \left(\delta c_{i}\ot\ell_{i}^r +r\,c_{i}\ot\ell_{i}^{r-1}\delta\ell_{i}\right),
\]
and by \eqref{eq:NonCrescenteShallow} we may extend $\ell_1 ,\dots,\ell_{d}$ to a
basis of $X\i\inv$ and write $\delta\ell_{i}=\sum_{j}a_{ij}\ell_{j}$.

If $r\ge3$ and $j\neq i$, the monomial $\ell_{i}^{r-1}\ell_{j}$ occurs only in the
$i$-th summand, with coefficient $r\,a_{ij}c_{i}$; since $c_{i}\neq0$ we get
$a_{ij}=0$, so $\delta\ell_{i}=a_{ii}\ell_{i}$. If $r=2$ and $d=\dim X\o$, then for
$i\neq j\le d$ the coefficient of $\ell_{i}\ell_{j}$ is
$2\left(a_{ij}c_{i}+a_{ji}c_{j}\right)$ and $c_{i},c_{j}$ are part of a basis, so
$a_{ij}=a_{ji}=0$; for $j>d$ the coefficient of $\ell_{i}\ell_{j}$ is
$2a_{ij}c_{i}$, whence $a_{ij}=0$. In both cases the coefficient of
$\ell_{i}^r $ gives $\delta c_{i}+r\,a_{ii}c_{i}=0$, so every element of the kernel
has the form
\[
\delta\ell_{i}=a_{i}\ell_{i},\qquad\delta c_{i}=-r\,a_{i}c_{i},
\]
which is the infinitesimal direction of the action of
$\gl_{i}\cdot\left(c_{i},\ell_{i}\right)=\left(\gl_{i}^{-r}c_{i},\gl_{i}\ell_{i}\right)$,
that is of $\g\left(X\right)\cdot w$.
\end{proof}

We recall that a morphism of schemes $f:T\to S$ is a monomorphism if the diagonal morphism $\Delta _f:T\to T\times _S T$ is an isomorphism. Equivalently $f_ Z:T\left( Z\right)\to S\left( Z\right)$ is injective for every $\CC$-scheme $Z$, and the following Lemma is a translation of the classical fact that an injective morphism which is unramified is indeed a monomorphism.

\begin{lem}\label{lem:MonomorfismoNonRamificato}
Let $X$ be an algebraic network that satisfies \eqref{cond:ReteDocile}. If every $w\in\W\c$ is a good parameter, then $\varphi=\phi_X|_{\Om\c\left(X\right)}$ is a monomorphism if and only if it is injective.
\end{lem}

\begin{proof}
Proposition
\ref{prop:BuonoEquivalenze} applies to $X$.  Clearly a monomorphism is injective
on closed points.  Conversely, every $w\in\W\c$ being good, the morphism
$\varphi$ is unramified, so its diagonal
\[
\Delta_{\varphi}:\Om\c\left(X\right)\longrightarrow\Om\c\left(X\right)\times_{\F\left(X\right)}\Om\c\left(X\right)
\]
is an open immersion; moreover $\varphi$ is separated since $\Om\c\left(X\right)$
is a variety, so $\Delta_{\varphi}$ is a closed immersion as well.  Its
complement is therefore open, and if it were non-empty it would contain a closed
point, which by the Nullstellensatz is a pair of $\CC$-points $x\neq y$ with
$\varphi\left(x\right)=\varphi\left(y\right)$, contradicting injectivity.  Hence
$\Delta_{\varphi}$ is an isomorphism.
\end{proof}

\begin{prop}\label{prop:ShallowMonomorfismo}
Assume \eqref{eq:NonCrescenteShallow} and either $r\ge3$, or $r=2$ and
$d=\dim X\o$. Then
\begin{enumerate}
\item every $w\in\W\c$ is a good parameter and
$\varphi=\phi_X\vert_{\Om\c\left(X\right)}$ is injective;
\item $\varphi$ is a monomorphism;
\item if $p\in\E\left(X\right)$, then $\varphi$ is not finite over $p$: there is
$\left[w\right]\in\varphi^{-1}\left(p\right)$ such that
$\OO_{\Om\c\left(X\right),\left[w\right]}$ is not a finite
$\OO_{\F\left(X\right),p}$-module.
\end{enumerate}
\end{prop}

\begin{proof}
Part (1) is Propositions \ref{prop:ShallowBuono} and \ref{prop:ShallowIniettivo}. 
For (2), the parameter space
$\W=\Hom\left(X\i,\N\right)\times\Hom\left(\N,X\o\right)$ is a vector space, hence smooth, and $\W\c$ is non-empty: in view of \eqref{eq:NonCrescenteShallow} we may choose $\ell_1 ,\dots,\ell_{d}$ linearly independent, since $d\le\dim X\i$, and non-zero $c_1 ,\dots,c_{d}$ spanning $X\o$, since $d\ge\dim X\o$. Being fully connected, $X$ satisfies \eqref{cond:ReteDocile} by Corollary \ref{cor:OmLiscioTotalmenteConnesse}, so Lemma \ref{lem:MonomorfismoNonRamificato} applies and, $\varphi$ being injective, it is a monomorphism.
For (3), $X$ satisfies \eqref{cond:ReteDocile} by (2), so Theorem \ref{thm:PuntualeProprio} applies; moreover by (1) the point $p$ is rigid, since $\varphi$ is injective and the whole fibre lies in $\Om\c\left(X\right)$, and all the parameters over $p$ are good.  Were $\OO_{\Om\c\left(X\right),\left[w\right]}$ finite over $\OO_{\F\left(X\right),p}$, that theorem would give the smoothness of $\F\left(X\right)$ at $p$, contradicting $p\in\E\left(X\right)$.
\end{proof}

\begin{rem}\label{rem:ShallowCondizionale}
In the terminology of Section \ref{Sec:Fonti}, Proposition \ref{prop:ShallowMonomorfismo} excludes the first two sources of exceptional singularities for these architectures: only non-properness is left. Whether $\E\left(X\right)=\e$ for all of them is therefore exactly the question of whether $\gs_{d}\left(\PP\left(X\o\right)\times v_{r}\PP\left(X\i\inv\right)\right)$ acquires a branch from the boundary at a point admitting a decomposition by $d$ linearly independent linear forms. Theorem \ref{thm:ListaShallow} answers it for $d\le2$, and for $d=3,4$ in the cases listed there; Remark \ref{rem:SottosecantiInDarch} says that no singularity produced by the present state of the art can answer it negatively. We do not know what happens for larger $d$.
\end{rem}

\begin{rem}\label{rem:CosaManca}
The hypothesis $d\le n$ in Proposition \ref{prop:ShallowIniettivo}~(1), (2) and in Proposition \ref{prop:ShallowBuono} is what makes the arguments work, since it forces
$\ell_{1},\dots,\ell_{d}$ to be linearly independent and lets us treat them as
part of a coordinate system.  In the wide range $d>n$, which is the one of
interest for expressivity, a full parameter only has pairwise non-proportional
$\ell_{i}$, and both proofs break down.  That the $\ell_{i}$ are pairwise
non-proportional at all is due to Definition \ref{def:LuogoDegenere}: had we
only required them to span $X\i\inv$, taking $\ell_1 =\ell_2$ and varying $c_1$
with $c_1 +c_2$ fixed would produce a one-parameter family of full parameters
lying in pairwise distinct orbits and realising the same function, so that
$\phi_X$ would fail to be injective on $\Om\c\left(X\right)$ for a reason having
nothing to do with the geometry of the secant variety.

Proposition \ref{prop:ShallowIniettivo}~(3) covers the wide range in sufficiently
large degree, but it does not pair with Proposition \ref{prop:ShallowBuono},
whose differential argument still requires $d\le n$; therefore Proposition
\ref{prop:ShallowMonomorfismo} does not extend to the wide range.

What remains open is the range of small $r$ and large $d$: generic identifiability of Segre-Veronese varieties is known in wide ranges \cite{Bal,CM23}, but what is needed here is identifiability at every non-redundant full point, which is a much stronger requirement. For a shallow network and a sufficiently general origin-passing polynomial activation of degree $r\ge7d^{2}$, it is proved in \cite[Remark 4.2]{KMMM26} that the weights are identifiable if and only if they are not a subnetwork.
\end{rem}

\subsection{The list}

In this section we apply the machinery we developed before and try to classify shallow fully connected networks with $\E\neq \e$. Let us fix some notations: for a fully connected network
\[
X:X\i\to\N\to X\o
\]
we let $n=\dim X\i$, $k=\dim X\o$, $d=\dim\N$ and activation $r$. Let 
\[
\M_r =\M_r \left(X\o, X\i\inv\right)=\PP\left(X\o\right)\times v_{r}\PP\left(X\i\inv\right),
\]
so that the projective neurovariety associated with the network is $\X=\gs_d \left(\M_r\right)$ whose projective geometry we studied in \cite{Gra26}. 

In view of this description it will be convenient to move the preceding apparatus to the projective setting. To do this we note that $\F\left(X\right)$, $\Phi_X\left(\W\right)$, $\Phi_X\left(\W\c\right)$ and $\Da\left(X\right)$ are stable under the scalar action on $X\o\ot\Sym^{R}\left(X\i\inv\right)$ and that the origin lies in $\Da\left(X\right)$. Moreover, being $\F\left(X\right)\setminus \{0\}\to\X$ a $\mathbb{G}_m$-torsor, we see that $\F\left(X\right)\setminus \{0\}$ is smooth at $f$ if and only if $\X$ is smooth at $\left[f\right]$. Hence it makes sense to define 
\[
\Da\left(\X\right)=\PP\left(\Da\left(X\right)\setminus\{0\}\right)
\qquad
\E\left( \X\right)=\left( \Sing \X \cap \PP\left( \Phi_X \left(\W\right)\setminus\{0\}\right)\right)\setminus\Da\left(\X\right)
\]
whence $\E\left(X\right)=\e$ if and only if $\E\left( \X\right)=\e$.

\begin{rem}\label{rem:SottosecantiInDarch}
When $\X$ is not the entire space, the subvariety $\gs_{d-1} \left(\M_r\right)\su \X$ consists entirely of singular points for simple geometric reasons \cite[Proposition 1.2.2(4)]{Rus16}. Since $\gs_{d-1} \left(\M_r\right)$ consists of limits of points that use $d-1$ neurons or less, we have $\gs_{d-1} \left(\M_r\right)\su\Da\left(\X\right)$.

We rule out another source of singularities: in \cite{FH25} Furukawa and Han describe, inside $\X$, the union
\[
\Sigma_{d,r}\left(m\right)=\bigcup_{U\in\Gr\left(m+1,X\i\inv\right)}\gs_{d}\left(\M_r \left(X\o,U\right)\right):
\]
it is a closed \cite[Remark 11(b)]{FH25} and irreducible subset, called the $m$-subsecant locus, and it is non-trivial whenever $m\le \min\{d,n\}-2$. It is the space of tensors $\sum_{i=1}^d c_{i}\ot\ell_{i}^r $ whose linear forms $\ell_{i}$ lie in a subspace of dimension $m+1$. In that paper the authors identify a number of cases in which  
\[
\Sigma_{d,r}\left(m\right) \su \Sing \X\setminus \gs_{d-1}\left(\M_r\right)
\]
hence providing a source of non-trivial singular points. In the case when $n\ge d$, in particular when the network satisfies \eqref{eq:NonCrescenteShallow}, then for every subspace $U\su X\i\inv$ of dimension $m+1\le d-1$, every point of $\gs_{d}\left(\M_r \left(X\o, U\right)\right)$ is a limit of tensors $\sum_{i=1}^d c_{i}\ot\ell_{i}^r $ whose linear forms $\ell_{i}$ lie in a fixed subspace of dimension $m+1\le d-1$; for such a tensor $w_1 $ has rank at most $d-1$, which is not maximal, so the parameter is not full and hence $\Sigma_{d,r}\left(m\right) \su \Da\left(\X\right)$. In particular for a fully connected shallow network $X$, if $\nu\i=\min\left\{ n,d\right\} -2$, $\nu\o=\min\left\{ d,k\right\} -2$, then
\begin{equation}\label{eq:DecomposizioneDaProiettivo}
    \Da\left(\X\right)=\gs_{d-1}\left(\M_r\right)\ \cup\ \Sigma_{d,r}\left(\nu\i\right)\ \cup\bigcup_{ V\in\Gr\left(\nu\o+1,X\o\right)} \gs_{d}\left(\M_r\left(V,X\i\inv\right)\right),
\end{equation}
the last term being empty when $\min\left\{ d,k\right\}\le1$. To deduce \eqref{eq:DecomposizioneDaProiettivo}, write
\[
\W\setminus\W\c=A\cup B\cup C\cup D,
\]
where
\[
A=\left\{\rk w_1<\min\{n,d\}\right\},\quad B=\left\{\rk w_2<\min\{d,k\}\right\},\quad
C=\bigcup_j\{c_j=0\},\]\[D=
\bigcup_i\{\ell_i=0\}
\ \cup\
\bigcup_{i<j}\{\ell_i\wedge\ell_j=0\}.
\]

Note that every tensor in
$\Phi_X(D)$ admits a decomposition with at most $d-1$ summands, in fact if $\ell_j=\lambda\ell_i$, then $c_i\ot\ell_i^r+c_j\ot\ell_j^r
=
\left(c_i+\lambda^r c_j\right)\ot\ell_i^r$, while a term with $\ell_i=0$ vanishes. Since the same
is true for $C$, and every decomposition with at most $d-1$ summands can be padded by a vanishing coefficient, we obtain
\[
\overline{\Phi_X(C\cup D)}
=
\widehat{\gs_{d-1}\left(\M_r\right)}.
\]

A parameter lies in $A$ exactly when its linear forms span a subspace of dimension at most $\nu\i+1$, so
$\overline{\Phi_X(A)}$ is the affine cone over $\Sigma_{d,r}\left(\nu\i\right)$. Similarly,
$\overline{\Phi_X(B)}$ is the affine cone over 
\[
\bigcup_{V\in\Gr\left(\nu\o+1,X\o\right)} \gs_d\left(\M_r\left(V,X\i\inv\right)\right).
\]
The latter union is closed, since it can be written as an intersection
\[
\X\cap\left\{\left[F\right]\in\PP\left( X\o\ot\Sym^r \left(X\i\inv\right)\right)\,\vert\rk F\le\nu\o+1\right\}
\]
of $\X$ with a determinantal variety.
All three resulting subsets are cones containing the origin, and projectivising gives \eqref{eq:DecomposizioneDaProiettivo}.
\end{rem}

We gather here a collection of cases that we could find in which $\E\left(\X\right)=\e$.

\begin{thm}\label{thm:ListaShallow}
Let $X:X\i\to\N\to X\o$ be a fully connected shallow network of activation
degree $r$, set $n=\dim X\i$, $k=\dim X\o$, $d=\dim\N$, and let
$\X=\gs_{d}\left(\M_r\right)\su\PP^{N}$ with $N=k\binom{n+r-1}{r}-1$. Then
$\E\left(\X\right)=\e$ in each of the following two situations.
\begin{enumerate}
\item $\X$ is smooth. This happens in particular when
 \begin{enumerate}
 \item $d=1$, since then $\X=\M_r$ is a Segre-Veronese variety;
 \item $\X=\PP^{N}$, which holds whenever $\X$ has the expected dimension and
 $d\left(n+k-1\right)\ge k\binom{n+r-1}{r}$, and in particular when $r=2$,
 $k=2$ and $d\ge n$; when $k=1$ and $d\ge g\left(n,r\right)$, the
 Alexander-Hirschowitz threshold \cite{AH95}; and when $d\ge\binom{n+r-1}{r}$.
 \end{enumerate}
\item $\Sing\X\su\Da\left(\X\right)$. This happens when
 \begin{enumerate}
 \item $r=1$;
 \item $r=2$ and $k=1$;
 \item $k=1$ and $n=2$;
 \item $d=2$;
 \item $k=1$, $d=3$ and $n\ge3$;
 \item $k=1$, $d=4$, $r\ge3$ and $n\ge4$.
 \end{enumerate}
\end{enumerate}
\end{thm}

\begin{proof}
In case (1) clearly $\E\left(\X\right)=\e$. For (1a), $\gs_{1}\left(\M_r\right)=\M_r$ is smooth. For (1b), $\PP^{N}$ is smooth, and the numerical criterion is the translation of fillingness, the expected dimension of $\X$ being $\min\left\{ d\left(n+k-1\right)-1,N\right\} $. The secant variety $\gs_{d}\left(\PP^{1}\times v_{2}\PP^{n-1}\right)$ is non-defective (see \cite[p. 756]{OV25}), and there the inequality reads $d\left(n+1\right)\ge2\binom{n+1}{2}=n\left(n+1\right)$, that is $d\ge n$; the threshold $g\left(n,r\right)$ is by definition the least $d$ for which $\gs_{d}\left(v_{r}\PP^{n-1}\right)$ fills its ambient space; and if $d\ge\binom{n+r-1}{r}$ the $r$-th powers $\ell_{1}^{r},\dots,\ell_{d}^{r}$ can be taken to be a basis of $\Sym^{r}\left(X\i\inv\right)$, so that $X$ is filling as well \cite[Proposition 3.7]{KLW24}.

In case (2) we may assume that $\X$ is singular.

For (2a) the variety $\M_{1}=\PP\left(X\o\right)\times\PP\left(X\i\inv\right)$ is a Segre variety and $\X$ is the locus of matrices of rank at most $d$; since  $\X\neq\PP^{N}$ we have $d<\min\left\{ n,k\right\} $, and $\Sing\X=\left\{ \rk\le d-1\right\} =\gs_{d-1}\left(\M_{1}\right)$; compare Example \ref{ex:LineareEVuoto}.

For (2b) we have $\X = \left\{ \left[q\right]\,\vert\,\rk q\le d\right\} \su \PP\Sym^{2}\left(X\i\inv\right)$; since $\X\neq\PP^{N}$ we have $d<n$, and $\Sing\X=\left\{\left[q\right]\,\vert\,\rk q\le d-1\right\} =\gs_{d-1}\left(\M_{2}\right)$.

For (2c) we have $\X=\gs_{d}\left(v_{r}\PP^{1}\right)\su\PP^{r}$, which fills $\PP^{r}$ as soon as $2d\ge r+1$; since $\X\neq\PP^{r}$ we have $r\ge2d$, so the embedding $\OO_{\PP^{1}}\left(r\right)$ is $2d$-very ample and \cite[Proposition 2.3]{Ver08} gives
$\Sing\X=\gs_{d-1}\left(v_{r}\PP^{1}\right)$, the reverse inclusion being Remark \ref{rem:SottosecantiInDarch}.

For (2d) recall that $\X$ is singular, and we may assume $r\ge2$, the case $r=1$ being (2a). If $k=1$ we may also assume $r\ge3$, the case $r=2$ being (2b), and \cite[Theorem 3.3(iii)]{Kan99} gives $\Sing\X=\M_r=\gs_{d-1}\left(\M_r\right)$, the first term of \eqref{eq:DecomposizioneDaProiettivo}. If $k\ge2$ then $\M_r$ is a Segre-Veronese variety of two factors with $\mathbf a=\left(1,r\right)$ and $\mathbf b=\left(k-1,n-1\right)$ in the notation of \cite{KMZ20}. Unless $r=2$ and $k\ge3$, \cite[Theorem 5.8(S2)]{KMZ20} applies and, by the discussion preceding \cite[Corollary 5.9]{KMZ20}, either $\X$ is filling, which is excluded, or $\Sing\X=\M_r$ again. If instead $r=2$ and $k\ge3$, then \cite[Theorem 5.8(1)(i)]{KMZ20} gives a single component and \cite[Corollary 5.10]{KMZ20} identifies it as $\PP\left(X\o\right)\times\gs_{2}\left(v_{2}\PP\left(X\i\inv\right)\right)$, which is the third term of \eqref{eq:DecomposizioneDaProiettivo} since $\nu\o=\min\left\{ 2,k\right\} -2=0$; equivalently, it is the locus in $\X$ of tensors with one-dimensional output span. In every case $\Sing\X\su\Da\left(\X\right)$.

For (2e) we may assume $r\ge3$, the case $r=2$ being (2b).  It is shown in \cite[Theorem 2.1]{Han18} that $\Sing\gs_{3}\left(v_{r}\PP^{n-1}\right)$ equals, as a set, $\gs_{2}\left(v_{r}\PP^{n-1}\right)$, except for $r=4$ and $n\ge4$, where it is the $1$-subsecant locus $\Sigma_{3,4}\left(1\right)$.

For (2f), in view of \cite[Theorem 5(i)]{FH25} the variety $\gs_{4}\left(v_{r}\PP^{n-1}\right)$ is smooth at every point outside $\Sigma_{4,r}\left(2\right)\cup\gs_{3}\left(v_{r}\PP^{n-1}\right)$. The hypothesis $n\ge4$ gives $\nu\i=\min\left\{ n,4\right\} -2=2$, so again both loci lie in $\Da\left(\X\right)$.
\end{proof}

Unfortunately the conjecture that $\E\left(\X\right)=\e$ for every shallow network is too optimistic, as the following example from \cite{AC20} shows.

\begin{ex}\label{ex:ShallowEccezionale}
Let $X:\CC^3 \to\CC^{14}\to\CC$ with activation degree $r=8$, so that $\X=\gs_{14}\left(v_{8}\left(\PP^2 \right)\right)$ whose generic rank is $15$. Note that this network does not satisfy \eqref{eq:NonCrescenteShallow}. Here $\W\c$ consists of the parameters whose linear forms
$\ell_1 ,\dots,\ell_{14}$ span $\left(\CC^3 \right)\inv$, are non-zero and are pairwise
non-proportional, and whose coefficients $c_1 ,\dots,c_{14}$ are all non-zero. Since the output is scalar, in view of \eqref{eq:DecomposizioneDaProiettivo} we have
\[
\Da\left(\X\right)=\left\{ \mbox{forms with at most two essential variables}\right\} \cup\gs_{13}\left(v_{8}\left(\PP^2 \right)\right),
\]
the first term being the closure of the image of the parameters with $\rk w_1 \le2$ and the second the closure of the image of those with a vanishing outgoing coefficient or a zero or redundant incoming linear form.

In \cite[Section 4.2]{AC20}, Angelini and Chiantini construct a ternary octic $T$ admitting two non-redundant decompositions of length $14$, disjoint by \cite[Claim 4.3]{AC20}.They prove in \cite[Remark 4.10]{AC20} that the corresponding point of $\gs_{14} \left(v_{8}\left(\PP^{2}\right)\right)$ is non-normal, hence singular. Write $T=\sum_{i=1}^{14}\ell_i^{8}$ for one of the normalised decompositions. For its support $A=\{[\ell_1],\dots,[\ell_{14}]\}$ they verify that $h_A \left(4\right)=\dim \left\langle \ell_1 ^4 ,\dots, \ell_{14}^4\right\rangle=14$ \cite[Section 4.2]{AC20}. The middle catalecticant
\[
\operatorname{Cat}_{4,4}\left(T\right):
\Sym^{4}\left(\CC^{3}\right) \longrightarrow
\Sym^{4}\left(\left(\CC^{3}\right)\inv\right)
\]
has $\operatorname{Im}\operatorname{Cat}_{4,4}\left(T\right) = \left\langle\ell_1^4,\dots,\ell_{14}^4\right\rangle $ hence $\rk\operatorname{Cat}_{4,4}\left(T\right) = 14$. On the other hand, the $(s+1)\times(s+1)$ minors of every catalecticant vanish on $\gs_s \left(v_8\left(\PP^2\right)\right)$ \cite[Section 5.6.2]{BCCGO18}; equivalently,
\[
\left[F\right] \in \gs_s\left(v_8\left(\PP^2\right)\right) \quad \Longrightarrow \quad
\rk \operatorname{Cat}_{4,4}\left(F\right)\leq s.
\]
Consequently $\left[T\right]\notin\gs_{13}\left(v_8\left(\PP^2\right)\right)$. 

We show now that $T$ has 3 essential variables, and to see this we recall that this coincides with the rank of the first catalecticant map, hence we show that $\rk\mathrm{Cat}_{1,7}\left(T\right)=3$. Note that this is equivalent to $\mathrm{Cat}_{1,7}\left(T\right)$ being injective, hence we suppose that 
\[
\mathrm{Cat}_{1,7}\left(T\right)\left(c\right)=8\sum \ell_i \left(c\right) \ell_i ^7=0.
\]
Since the Hilbert function is non-decreasing, the $\ell_{i}^{7}$ are linearly independent and hence a non-zero $c$ in the kernel would force all fourteen points to lie on the line $c^{\perp}$, but $h^0 \left(\PP^1,\OO_{\PP^1}\left(4\right)\right)=5$, giving $h_A \left(4\right)\le 5< 14$. It follows that $\ker \mathrm{Cat}_{1,7}\left(T\right)=0$, in particular $T$ has 3 essential variables hence $\left[T\right]\notin\Da\left(\X\right)$ and $\left[T\right]\in\E\left(\X\right)$.

The displayed decomposition of $T$ is full: its support consists of fourteen distinct points spanning $\PP^2$, and all its coefficients are non-zero, hence $\W\c\neq\e$. Since $X$ is fully connected, Corollary \ref{cor:OmLiscioTotalmenteConnesse} shows that it satisfies \eqref{cond:ReteDocile}, and Theorem \ref{thm:Singolarita} applies to it.  Since $T\notin\Da\left(X\right)$, we have $T\in C\left(X\right)\cup N\left(X\right)\cup P\left(X\right)$. By \cite[Claim 4.5]{AC20} and the verification carried out there, $T$ is computed by exactly two sets of $14$ points; since the eighth powers of the elements of each set are linearly independent, the coefficients are determined by the set, so each of the two gives a single $\G\left(X\right)$-orbit and $\phi_X^{-1}\left(T\right)$ consists of exactly two points.  It is therefore disconnected so that $T$ is a node.
\end{ex}

\section{Normal forms, conjectures and open questions}\label{Sec:Aperti}

\subsection{Normal forms}\label{Sec:formaNormale}

In this section we propose a canonical way to reduce a fully connected network in order to remove redundant layers while keeping the network fully connected and the space of functions unchanged; this will help in formulating Conjecture \ref{conj:Soglia} and Problems \ref{prob:Eccesso} and \ref{prob:AltreFonti}. We show that the sets $\Da$ are non-increasing and $\E$ are non-decreasing during this process.

Let $X$ be fully connected and write $m_{t}=\dim\N_{t}$.

\begin{construction}[contraction of a linear block]\label{constr:ContrazioneLineare} Let $a\le b$ be such that $r_{a}=\dots=r_{b}=1$, so that the portion $\N_{a-1}\to\N_{a}\to\dots\to\N_{b}\to\N_{b+1}$ is purely linear, and set $\rho=\min\left\{ m_{a-1},m_{a},\dots,m_{b},m_{b+1}\right\} $. Replace the whole chain by a single hidden layer $M$ with $\dim M=\rho$ and linear activation. 
\end{construction}

\begin{construction}[contraction of a one-dimensional chain] \label{constr:ContrazioneUnidim}
Let $a\le b$ be such that $m_{a}=\dots=m_{b}=1$. Replace the chain by a single one-dimensional hidden layer with activation degree $r_{a}r_{a+1}\cdots r_{b}$.
\end{construction}

\begin{defn}\label{def:FormaNormale}
$X$ is in normal form if neither of Constructions \ref{constr:ContrazioneLineare} and \ref{constr:ContrazioneUnidim} changes it, that is if
\begin{enumerate}
\item there is no $t$ with $r_{t}=r_{t+1}=1$;
\item there is no $t$ with $m_{t}=m_{t+1}=1$;
\item $m_{t}\le\min\left\{ m_{t-1},m_{t+1}\right\}$ for every $t$ with $r_{t}=1$.
\end{enumerate}
We write $X\n$ for any network in normal form obtained from $X$ by iterating these reductions.
\end{defn}
Note that (iii) is needed because Construction \ref{constr:ContrazioneLineare} applies to a single isolated linear layer as well, and shrinks it to the bottleneck width.

To prove uniqueness we need to invoke Newman's Lemma \cite[Theorem 1.2.1]{Ter03}, which can be stated in the following form
\begin{thmstar}[Newman's Lemma]
    A terminating abstract reduction system is confluent if it is locally confluent.
\end{thmstar}

\begin{lem}\label{lem:FormaNormaleUnica}
Every fully connected network reduces to a network in normal form, and the result does not depend on the order in which Constructions \ref{constr:ContrazioneLineare} and \ref{constr:ContrazioneUnidim} are applied. We denote it by $X\n$.
\end{lem}

\begin{proof}
Clearly the application of the constructions \ref{constr:ContrazioneLineare} and \ref{constr:ContrazioneUnidim} eventually gives a network in normal form. To see this formally one notes that every non-trivial reduction strictly decreases the pair $\left(L,\sum_{t=1}^{L-1}m_t\right)$ in the lexicographic order. To apply Newman's Lemma we need to check local confluence, that is: given two 1-step reductions of $X$
\[
X_1	\leftarrow X\rightarrow X_2
\] 
there exist sequences of reductions (that may be the identity)
\[
X_1\rightarrow	 Y\leftarrow X_2
\]
to a common target network. The only case to be checked is the case when there is an overlap between a block of 1-dimensional layers and of linear layers. If one is contained in the other, contracting the smaller block  preserves the condition defining the larger one. Indeed, contracting a one-dimensional block contained in a linear block produces a linear layer, since the product of its activation degrees is $1$; similarly, contracting a linear block contained in a one-dimensional block  produces a one-dimensional layer, since the minimum of its widths is $1$. Thus, in either order, the larger block can subsequently be contracted. Suppose finally that the two blocks overlap without either containing the other. We suppose that
\[
I=\left[a,b\right],\qquad J=\left[c,d\right], \qquad a<c\leq b<d,
\]
where $I$ is linear and $J$ is one-dimensional, the other case is treated similarly. We have
\[
r_a=\cdots=r_b=1, \qquad m_c=\cdots=m_d=1.
\]
Contracting $I$ first produces a one-dimensional linear layer: its activation degree is $1$, and its width is $1$ because the minimum defining its width contains one of the one-dimensional layers of $J$. This layer and the part of $J$ which remains form a one-dimensional block, which can be contracted. Conversely, contracting $J$ first produces a one-dimensional layer. The part of $I$ which remains is still linear and can be contracted
using this new one-dimensional layer. The resulting linear layer has width $1$. The two consecutive one-dimensional layers can then be contracted. In both cases the final layer has width 
\[
\min\left\{m_{a-1},m_a,\dots,m_d,m_{d+1}\right\}=1
\]
and activation degree $r_a r_{a+1}\cdots r_d$, hence the two reductions admit a common reduct. 
\end{proof}

Both constructions leave the neurovariety and the image of the realisation map
unchanged, so they are harmless from the point of view of $\F\left(X\right)$. 

\begin{lem}\label{lem:EContenutoInEnorm}
If $Y$ is obtained from $X$ by one of Constructions
\ref{constr:ContrazioneLineare} and \ref{constr:ContrazioneUnidim}, then
$\Da\left(Y\right)\su\Da\left(X\right)$ and hence
$\E\left(X\right)\su\E\left(Y\right)$. In particular
$\E\left(X\right)\su\E\left(X\n\right)$.
\end{lem}

\begin{proof}
   Note that we have $\F\left( X \right) = \F\left( Y \right)$ and also $\Phi_X\left(\W_X\right)=\Phi_Y\left(\W_Y\right)$. To prove that $\Da\left( Y \right)\subseteq \Da\left( X \right)$ it is enough to show that every function admitting a degenerate representation in $Y$ also admits a degenerate representation in $X$. Consider first the contraction of a linear block 
\begin{equation}\label{eq:BloccoEstesoLineare}
A\longrightarrow N_a\longrightarrow\cdots\longrightarrow N_b\longrightarrow B
\end{equation}
and let $\rho=\min\{\dim A,\dim N_a,\ldots,\dim N_b,\dim B\}$, then the contracted block is 
\begin{equation}\label{eq:BloccoContrattoLineare}
A\xrightarrow{u}M\xrightarrow{v} B, \qquad \dim M=\rho. 
\end{equation}
Choose coordinate embeddings and projections
\[
i_j:M\hookrightarrow N_j,\qquad p_j:N_j\twoheadrightarrow M
\]
for $j=a,\ldots,b$, with $p_ji_j=\mathrm{id}_M$, and set $\widetilde u=i_a\circ u$, $\widetilde w_j=i_{j+1}\circ p_j$ for $a\le j<b$ and $\widetilde v=v\circ p_b$. Then
\[
\begin{aligned}
\widetilde v\circ\widetilde w_{b-1}\circ\cdots
\circ\widetilde w_a\circ\widetilde u
&=
v\circ p_b\circ i_b\circ p_{b-1}\circ\cdots
\circ i_{a+1}\circ p_a\circ i_a\circ u\\
&=v\circ u,
\end{aligned}
\]
so every factorisation through \eqref{eq:BloccoContrattoLineare} lifts to a factorisation through \eqref{eq:BloccoEstesoLineare} without changing the realised function. Suppose that a parameter of $Y$ is degenerate. If the degeneracy occurs outside the contracted block, we lift the block as above and leave all the other matrices unchanged, so the lifted
parameter is still degenerate. 

Suppose instead that it occurs inside the contracted block and consider now the conditions involving $u$ and $v$. Since $i_a$ is injective and $p_b$ is surjective,
\[
\rk\left(i_a\circ u\right)=\rk u,
\qquad
\rk\left(v\circ p_b\right)=\rk v,
\]
therefore a rank drop of either $u$ or $v$ remains such after lifting. A zero column of $u$, when $A$ is a hidden layer, remains zero after composition with $i_a$. A zero row or two proportional rows of $v$, when $B$ is a hidden layer, remain zero or proportional after composition on the right with $p_b$. The rows of $u$ occur among the rows of $i_a\circ u$, while the columns of $v$ occur among the columns of $v\circ p_b$; hence zero or proportional rows of $u$, and zero outgoing columns of $v$, are also preserved. If one of the spaces $N_j$ has dimension larger than $\rho$, the chosen coordinate embeddings and projections may introduce additional zero rows or columns, but this only makes the lifted parameter degenerate. Thus every degeneracy of the contracted representation either survives the lift or the lift acquires an additional degeneracy, in all cases the realised function is unchanged.

Consider now the contraction of a one-dimensional chain
\[
A\longrightarrow N_a\longrightarrow\cdots\longrightarrow N_b \longrightarrow B,
\qquad \dim N_a=\cdots=\dim N_b=1.
\]
Put $R=r_a r_{a+1}\cdots r_b$, then the contracted block is
\[
A\xrightarrow{u}M\xrightarrow{v}B, \qquad \dim M=1,
\]
with activation degree $R$. Under the natural identifications $M\simeq N_a\simeq\cdots\simeq N_b$, the adjacent maps $u$ and $v$ are left unchanged, while all internal maps are the identity. Consequently every rank drop, zero incoming row, zero outgoing column, or proportionality relation occurring in the contracted parameter is still present in the lifted parameter. The internal identity maps introduce no new obstruction. Hence every non-full parameter of $Y$ lifts to a non-full parameter of $X$ with the same realisation. We have therefore proved $\Phi_Y\left(\W_Y\setminus\W_Y\c\right) \su \Phi_X\left(\W_X\setminus\W_X\c\right)$ hence
\begin{equation}\label{eq:InclusioneDArch}
\Da\left(Y\right)\subseteq\Da\left(X\right).
\end{equation}
Since moreover $\F\left(X\right)=\F\left(Y\right)$ and $\Phi_X\left(\W_X\right)=\Phi_Y\left(\W_Y\right)$ the two networks have the same set
\[
S= \Sing\F\left(X\right)\cap\Phi_X\left(\W_X\right) = \Sing\F\left(Y\right)\cap\Phi_Y\left(\W_Y\right)
\]
and $\E\left(\bullet\right)=S\setminus\Da\left(\bullet\right)$, which, in view of \eqref{eq:InclusioneDArch}, gives
\[
\E\left(X\right)=S\setminus\Da\left(X\right) \su S\setminus\Da\left(Y\right)=\E\left(Y\right).
\]
\end{proof}

\begin{ex}\label{ex:InclusioneStretta}
The inclusion of Lemma \ref{lem:EContenutoInEnorm} can be strict, so $\E$ is not invariant under the reductions. Let
\[
X:\CC^2 \longrightarrow\CC^3 \longrightarrow\CC^2 \longrightarrow\CC^2 \longrightarrow\CC^2 ,
\qquad\rr=\left(1,2,2\right),
\]
\[
Y:\CC^2 \longrightarrow\CC^2 \longrightarrow\CC^2 \longrightarrow\CC^2 ,
\qquad\rr=\left(2,2\right),
\]
where $Y$ is the network of Section \ref{Sec:Esempio} and $X$ is obtained by prefixing a linear layer through $\CC^3 $. Condition (iii) of Definition \ref{def:FormaNormale} fails, and the linear contraction replaces $\CC^3 $ by a plane, giving
\[
X\n:\CC^2 \longrightarrow\CC^2 \longrightarrow\CC^2 \longrightarrow\CC^2 \longrightarrow\CC^2 ,
\qquad
r=\left(1,2,2\right).
\]
Every $A:\CC^2 \to\CC^2 $ factors as
\[
A=\begin{pmatrix}1&0&0\\0&1&0\end{pmatrix}\begin{pmatrix}A\\0\ \ 0\end{pmatrix},
\]
and although both factors have maximal rank, the first one has a vanishing third column, hence every function of $\F\left(X\right)$ admits a degenerate representation and $\Da\left(X\right)=\F\left(X\right)$, whence $\E\left(X\right)=\e$. On the other hand $\Phi_{X\n}\left(w_1 ,w_2 ,w_3 ,w_4 \right)=\Phi_{Y}\left(w_2 w_1 ,w_3 ,w_4 \right)$, so the three parameters $w,w^\prime,w^{\prime\prime}$ of Section \ref{Sec:Esempio}, prefixed by the identity, are full, realise the same $p$, and have the same tangent images as in $Y$; moreover $\F\left(X\n\right)=\F\left(Y\right)$ and a degenerate parameter of $X\n$ still makes the two output quartics proportional. Hence $p\in\E\left(X\n\right)$ and
\[
\E\left(X\right)=\e\subsetneq\E\left(X\n\right).
\]
\end{ex}

Because of Example \ref{ex:InclusioneStretta}, Conjecture \ref{conj:Soglia} below is stated for networks in normal form.

\subsection{A conjecture about the activation degree}

Section \ref{Sec:Generalizzazione} produces exceptional singularities out of a quadratic block, and Remark \ref{rem:SoloGradoDue} shows that the mechanism is specifically related to quadratic activation. There is in fact a heuristic reason to expect that nothing similar happens for large activation degrees.
Combining this with \cite[Theorem 5.1]{KMSSW26}, which excludes cusps as a source of exceptional singularities for monomial networks of large activation degree, we are led to the following.

\begin{conj}\label{conj:Soglia}
For every sequence of widths $\underline{m}$ there is $r_{0} = r_{0}\left(\underline{m}\right)$ such that every fully connected network in normal form with those widths and with $r_{t}\ge r_{0}$ for every $t$ satisfies $\E\left(X\right)=\e$.
\end{conj}

Conjecture \ref{conj:Soglia} and Theorem \ref{thm:Generalizzazione} together would say that the exceptional locus is a low-degree phenomenon, and it is not, as one might have guessed, related to depth.

\subsection{Open questions}

\begin{problem}\label{prob:Eccesso}
In every architecture of Remark \ref{rem:CasiVerificati} the span of the three tangent images exceeds $\dim\Om\left(X\right)$ by a quantity depending only on the width $m$ of the quadratic block, equal to $5m-8$ for $2\le m\le7$. Is it always $5m-8$ independently of the depth and of the remaining activation degrees?  Already the weaker statement that the excess is positive would make Theorem \ref{thm:Generalizzazione} unconditional.
\end{problem}

\begin{problem}\label{prob:AltreFonti}
All the exceptional singularities exhibited in this paper are nodes.  Does
there exist a fully connected monomial network in normal form with an
exceptional singularity which is not a node?
\end{problem}

\begin{rem}\label{rem:CosaSignificaPerLottimizzazione}
We close with the reason why the distinction matters outside algebraic geometry. The sources described in Section \ref{Sec:Fonti} have different consequences for optimisation.  Cusps are subnetworks, for monomial networks of large activation degree by \cite[Theorem 5.1]{KMSSW26}, and strict subnetworks are critically exposed in the sense of \cite[Theorem 4.3]{SMK26}: they are critical points of the training loss for a non-negligible set of datasets, which is the geometric content of the implicit bias towards sparse solutions.  The points of $\E\left(X\right)$ are of a different nature: the parameters realising them are full and, in the example of Section \ref{Sec:Esempio}, good, so nothing in the parametrisation signals that the function is singular.  What fails there is the global injectivity of $\phi_X$ on $\Om\c\left(X\right)$, that is the possibility of recovering the parameter, up to the symmetries of the architecture, from the function it computes.
\end{rem}




\newpage
\bibliographystyle{alpha}
\bibliography{References}

\end{document}